\documentclass[12pt]{article}
\usepackage{amsmath,amssymb,amsthm, amsfonts,xcolor,enumerate}
\usepackage{comment}
\usepackage{hyperref,cleveref}
\usepackage{graphicx}
\usepackage{soul}
\usepackage{url}
\usepackage[mathlines]{lineno}
\usepackage{dsfont} 
\usepackage{color}
\definecolor{red}{rgb}{1,0,0}

\definecolor{blue}{rgb}{0,0,1}

\definecolor{green}{rgb}{0,.6,0}

\definecolor{purp}{rgb}{.5,0,.5}

\usepackage{tikz}
\usetikzlibrary{arrows.meta} 
\usepackage{float}

\newtheorem{thm}{Theorem}[section]
\newtheorem{cor}[thm]{Corollary}
\newtheorem{lem}[thm]{Lemma}
\newtheorem{prop}[thm]{Proposition}

\newtheorem{obs}[thm]{Observation}

\theoremstyle{definition}
\newtheorem{rem}[thm]{Remark}

\theoremstyle{definition}
\newtheorem{defn}[thm]{Definition}

\theoremstyle{definition}
\newtheorem{ex}[thm]{Example}

\numberwithin{figure}{section}   
\numberwithin{equation}{section}  
\newcommand{\R}{\mathbb{R}}

\newcommand{\GL}[1]{\mathcal{G} ^ #1}

\newcommand{\GLG}{G^L}
\newcommand{\GLGsig}{G^{\sigma(L)}}
\newcommand{\GC}[2]{#1 ^ {\mathfrak #2}}
\newcommand{\PP}{\mathfrak{C}}
\newcommand{\pp}{\mathfrak{c}}

\newcommand{\MM}{\mathfrak{M}}
\newcommand{\mm}{\mathfrak{m}}
\newcommand{\II}{\mathfrak{I}}

\newcommand{\G}{\mathcal{G}}

\newcommand{\Span}{\operatorname{span}}
\newcommand{\sym}{\mathcal{S}}
\newcommand{\symp}{\mathcal{S}_+}

\newcommand{\A}{\mathcal{A}}

\newcommand{\by}{{\bf y}}
\newcommand{\bx}{{\bf x}}

\newcommand{\symm}{\operatorname{Symm}}
\newcommand{\Sp}{\operatorname{Sp}}
\newcommand{\spLA}{\operatorname{\mathfrak{sp}}}

\newcommand{\oml}{{\bf m}}
\newcommand{\soml}{{\bf m}_s}
\newcommand{\tp}{\operatorname{TP}}
\newcommand{\Rnn}{\R^{n\times n}} 
 
\newcommand{\Rn}{\R^{n}}

\newcommand{\trans}{^{\top}}
\newcommand{\tr}{\operatorname{tr}}

\newcommand{\mult}{\operatorname{mult}}

\newcommand{\Z}{\operatorname{Z}}
\newcommand{\ZC}{\Z_{\operatorname{C}}}
\newcommand{\Zell}{\Z_{\ell}}
\newcommand{\M}{\operatorname{M}}
\newcommand{\MC}{\M_{\operatorname{Sp}}}

\newcommand{\mtx}[1]{\begin{bmatrix} #1 \end{bmatrix}}

\newcommand{\spec}{\operatorname{spec}}

\newcommand{\nul}{\operatorname{null}}

\newcommand{\diag}{\operatorname{Diag}}
\newcommand{\spspec}{\operatorname{SPspec}}
\newcommand{\uvec}{\operatorname{vec}_\triangle}

\newcommand{\noi}{\noindent}
\newcommand{\x}{\times}
\newcommand{\lam}{\lambda}

\newcommand{\LS}{\Lambda^S}
\newcommand{\bit}{\begin{itemize}}
\newcommand{\eit}{\end{itemize}}
\newcommand{\ben}{\begin{enumerate}}
\newcommand{\een}{\end{enumerate}}
\newcommand{\beq}{\begin{equation}}
\newcommand{\eeq}{\end{equation}}
\newcommand{\bea}{\begin{eqnarray*}}
\newcommand{\eea}{\end{eqnarray*}}
\newcommand{\bean}{\begin{eqnarray}}
\newcommand{\eean}{\end{eqnarray}}
\newcommand{\bpf}{\begin{proof}}
\newcommand{\epf}{\end{proof}\ms}
\newcommand{\bmt}{\begin{bmatrix}}
\newcommand{\emt}{\end{bmatrix}}
\newcommand{\ms}{\medskip}

\newcommand{\beqa}{\begin{array}}
\newcommand{\eeqa}{\end{array}}
\newcommand{\OL}{\overline}

\newcommand{\lsb}{\left\{}
\newcommand{\rsb}{\right\}}
\newcommand{\wt}{\widetilde}
\newcommand{\wh}{\widehat}

\newcommand{\du}{\mathbin{\,\sqcup\,}}

\title{The Inverse Symplectic Eigenvalue Problem of a Graph}
\author{Himanshu Gupta\thanks{Department of Mathematics, IIT Gandhinagar, Palaj, GJ 382355, India (himanshugupta@iitgn.ac.in);
Department of Mathematics and Statistics, University of Regina, Regina, SK S4S0A2, Canada.} \and Leslie Hogben\thanks{American Institute of Mathematics, Pasadena, CA 91125, USA
(hogben@aimath.org); Department of Mathematics, Iowa State University,
Ames, IA 50011, USA; Department of Mathematics, Purdue University,
West Lafayette, IN 47907, USA.}\and Bryan Shader\thanks{Department of Mathematics \& Statistics, University of Wyoming, Laramie, WY 82071, USA (bshader@uwyo.edu)}\and Tony W. H. Wong\thanks{Department of Mathematics, Kutztown University of Pennsylvania, Kutztown, PA 19530, USA (wong@kutztown.edu)}}

\begin{document}
\maketitle
\begin{abstract} 
Symplectic geometry plays an increasingly important role in mathematics, physics, and applications.  The concept  of symplectic eigenvalues, whose existence for positive definite matrices is known as Williamson's Theorem,  naturally gives rise to inverse problems. 
We introduce the inverse symplectic eigenvalue problem for positive definite matrices described by a labeled graph and  solve it for several families and all labeled graphs of order  four. 
We develop various tools such as the Strong Symplectic Spectral Property (SSSP) together with its consequences, graph couplings
to describe collections of labelings of a graph that produce the same symplectic eigenvalues, and coupled graph zero forcing.  
We establish numerous results for symplectic positive definite matrices, including a sharp lower bound on the number of nonzero entries of such an irreducible  matrix (or equivalently, the number of edges in its graph).  This lower bound is a consequence of a  lower bound on the sum of the number of nonzero entries in an  irreducible positive definite matrix and its inverse.
\end{abstract}

\noi {\bf Keywords} symplectic matrices, symplectic eigenvalues of positive definite matrices, inverse problems, strong properties, zero forcing, graphs. 

\noi{\bf AMS subject classification} 15A18, 15B48, 15B57, 05C50.


\section{Introduction}\label{s:intro} 

\subsection{Motivation}\label{ss:motiv}

The \textit{Inverse Eigenvalue Problem for Graphs} (IEP-$G$) has been and continues to be  a very active area of research since its formal birth in 
the 1997 paper \cite{NU97}. The IEP-$G$ was motivated by the desire to understand how the fundamental modes of a vibrating system of $n$ objects
are constrained by the structure of the system.  The dynamics of such a vibrating system are determined by the spectra of an associated symmetric $n\times n$
 matrix, $A=[a_{ij}]$. The graph of $A$ (that is, the graph with vertices $1, \ldots, n$
 and an edge joining distinct $i$ and $j$ if and only if $a_{ij}\neq 0$) is determined by the interconnections 
 between the objects in the system. The IEP-$G$ asks: for a given graph $G$, what are the possible spectra of  symmetric 
 matrices with graph $G$?  See \cite{HLSbook} for a comprehensive treatment of the IEP-$G$. 

A gyroscopic system is a set of $p$ interconnected devices 
that uses a rapidly spinning wheel or light beam to measure, maintain, or control orientation and angular velocity based on the conservation of angular momentum. Using a state-space Hamiltonian formulation, the dynamic behavior of the gyroscopic system can be explained 
in terms of a multi-set of $p$ real numbers, called the  symplectic spectra, of an associated $2p\times 2p$  real positive definite matrix $N=[n_{ij}]$.
The definition of the  symplectic spectrum and its basic properties are given in \Cref{ss:Primer}. 
For more details about the mathematics of gyroscopic systems see \cite{L13}.   The interrelationships between the 
devices in the system determine the graphical structure of $N$; that is, for distinct $i$ and $j$,    $n_{ij} \neq 0$ if and only if there is an edge joining $i$ and $j$. 
This paper seeks to  understand what constraints the interrelationships between the devices place 
on the symplectic spectrum. More precisely, given a labeled graph on $2p$ vertices,  the \textit{inverse symplectic eigenvalue problem for $G$} (ISEP-$G$)  asks: which multi-sets of $p$ positive real numbers occur as the symplectic spectra of a positive definite matrix with graph $G$? While our main focus has been to build the theory of the  ISEP-$G$ in analogy with that done for the IEP-$G$, the gyroscopic model provides a grounding for the ISEP-$G$ similar to that which the vibrating system provides for the IEP-$G$.

In \Cref{ss:contributions}, 
we briefly 
describe the main contributions of this paper, without getting into the mathematical details, which are given in \Cref{ss:contrib2}. In  \Cref{ss:Primer}, we recall  known facts   about symplectic spectra that will be used throughout this work, and provide a brief history of their  development and usage.  \Cref{ss:prelim}  provides preliminary results about 
the ISEP-$G$. 
\Cref{ss:couplings} introduces the notion of coupled graphs, 
and \Cref{ss:notation} establishes the needed notation.

\subsection{Summary of contributions}
\label{ss:contributions}
The main contributions of this paper are the following. 
\begin{enumerate}
\renewcommand{\labelenumi}{$\circ$}
\item The introduction of  the Inverse Symplectic Eigenvalue Problem for Graphs  (\Cref{ss:motiv}) and the development of a new combinatorial 
concept, namely,  a coupling of a graph $G$, used to 
describe the set of labelings of $G$ that necessarily have the same symplectic spectra (\Cref{ss:couplings}).
\item Introduction and theoretical development 
of the strong symplectic spectral property (SSSP) --- a tool that allows one to solve 
related  ISEP-$G$s related to a matrix $A$ with the SSSP (\Cref{s:strong}).  
\item Foundational results for extreme symplectic spectra; in particular for the case where all  symplectic eigenvalues are equal (see \Cref{s:PD} and \Cref{s:ssa}). A sharp lower bound on the sum of the number of nonzero  entries in an irreducible positive definite matrix $M$ and the number of nonzero  entries in $M^{-1}$ is established, which implies a sharp lower bound on the number of edges in a connected  coupled graph that allows a  matrix with all symplectic eigenvalues equal.
\item Construction of many families of coupled graphs on $2p$ vertices  that are symplectic spectrally arbitrary; that is,  a coupled graph $G^\PP$  such that every multi-set of $p$ positive reals is the symplectic spectrum  of some matrix with  coupled graph $G^\PP$ (\Cref{s:ssa}).  
\item Development of a combinatorial
tool for coupled graphs, known as coupled zero forcing,  that gives an upper bound on the maximal multiplicity of a symplectic eigenvalue of a positive definite matrix whose labeled graph is described by the coupled graph. (\Cref{s:zeroforcing}).  
\item Complete resolution of 
the ISEP-$G$ for graphs on four vertices (\Cref{s:order4}). 
\end{enumerate}

\subsection{Primer on Symplectic Eigenvalues} 
\label{ss:Primer} 

Let 
$ 
\Omega_{n} = \left[ \begin{array}{rc} O_p & I_p \\ 
-I_p &O_p \end{array} \right]$,
where $n=2p$ and  $I_p$ (respectively, $O_p$) denotes the $p\times p$
identity (respectively, zero) matrix; when the order is clear from context, we may omit the subscript.  
The {\it symplectic group of order $n$} is denoted by $\Sp(n)$
and is the set of all $n \x n$ real matrices $S$ such that $
S\trans \Omega S=\Omega$.
Evidently, the symplectic matrices of order $n$
are those block matrices  of the form 
\[
\left[ 
\begin{array}{rc} 
S_{11} & S_{12} \\  
S_{21} & S_{22} \end{array}
\right], 
\]
where each block is $p\x p$, $S_{11}\trans S_{21}$, and $S_{12}\trans S_{22}$ are symmetric, and $S_{11}\trans S_{22}-S_{21}\trans S_{12}=I$. 
Examples of symplectic matrices  include
\begin{align}
\label{basic}
\mbox{ $\Omega, \quad \mtx{A & O\\O & (A^\top)^{-1}}$, \quad  and \quad  $\mtx{I & B\\O & I}$},
\end{align}
where $A$ and $B$ are $p \times p$ real matrices, $A$ is invertible, and $B$ is symmetric; the matrices in \eqref{basic} are called \emph{basic symplectic matrices}. 
The $n\times n$ symplectic matrices are closed under products, inverses, and transposes. It is known that 
every $n\times n$ symplectic matrix $S$ is a finite product of matrices  of the forms in \eqref{basic}. Indeed, it is shown in \cite{JLX22} 
 that  $S$ is a product of at most five basic symplectic matrices of the third type and/or their transposes.

Symplectic matrices arise geometrically and can be viewed as an analog of real orthogonal matrices.  The standard Euclidean space of dimension $m$ consists of $\mathbb{R}^m$ endowed with 
the bilinear form $\langle \bx, \by\rangle= \bx^\top I \by= \bx^\top \by$. An $m\times m$ orthogonal matrix
$Q$ is precisely a matrix that preserves this bilinear form (that is, 
$\langle Q\bx, Q\by\rangle= \langle \bx, \by \rangle$ for all $\bx, \by \in \mathbb{R}^m$).
The  standard \textit{symplectic vector space of dimension $n=2p$} consists of $\Rn$  endowed with  the bilinear form  $\langle\langle \bx , \by \rangle\rangle= \bx^\top\Omega \by$.  The matrix $S$ is symplectic if and only if 
$\langle\langle S\bx, S \by \rangle\rangle= \langle\langle \bx, \by \rangle\rangle$ for all $\bx, \by \in \mathbb{R}^{n}$; that is, if and only if 
the map $\bx \mapsto S\bx$ preserves the bilinear form $\langle\langle \, , \, \rangle \rangle$.
In short, symplectic matrices are to symplectic spaces as orthogonal matrices are to Euclidean spaces.

The classic spectral theorem for real symmetric matrices asserts that each symmetric matrix $A$ is 
orthogonally similar to a diagonal matrix whose entries are the eigenvalues of $A$.
{The analog of this result in the symplectic setting is based on Williamson's  work on matrix pencils \cite{W36} applied to 
real positive definite matrices, which by definition are symmetric. 
For an elementary proof of Williamson's Theorem, see \cite{I18}.}

\begin{thm}[Williamson's Theorem]
\label{th:Williamson} 
Let $N$ be an $n \times n$ real positive definite matrix with $n=2p$. 
Then there exists a symplectic matrix $S$ and a $p\times p$ diagonal matrix $D$ such that 
$ S\trans N S= \left[ \begin{array}{cc}
D & O \\
O & D \end{array} \right]$.
Moreover, the diagonal entries  of $D$ are  unique up to re-ordering.
\end{thm}

The $p$ diagonal entries of the diagonal matrix $D$ in Williamson's Theorem  are called the {\it symplectic eigenvalues of $N$}; the multi-set of $p$ symplectic eigenvalues of $N$ is called the \emph{symplectic spectrum of $N$} and is denoted by $\spspec(N)$.

Since $S\trans NS$ is congruent to $N$ and $N$ is positive definite, each symplectic eigenvalue 
is positive.  The \textit{multiplicity} of the symplectic eigenvalue $\lambda$ of $N$
is defined to be the number of occurrences of $\lambda$ on the diagonal of $D$. A symplectic eigenvalue is \textit{simple}
if its multiplicity is one.

Williamson's Theorem, when stated in the form of \Cref{th:Williamson}, is not true for all Hermitian positive definite  matrices (for a counterexample, see \cite{noWcplxPSD}); although there is a more complicated analog given in \cite{IN}. Thus, we restrict our attention to real positive definite matrices; future research may extend the results of this paper to the  complex Hermitian case.

Since  the rank of $D\oplus D$ is necessarily even,  it is impossible to have $S\trans MS=D\oplus D$  if $M$ is a real positive semidefinite matrix of odd rank and $S$ is symplectic.
Thus, not every $2p\times 2p$ {real} positive semidefinite matrix is symplectically congruent to  a matrix of the form $D \oplus D$ where $D$ is a diagonal matrix. Those that are have been  characterized in \cite{JM22}. 
 It is known that every symmetric matrix is symplectically congruent to  a direct sum of ``symplectically indecomposable'' matrices \cite{LM74}, and future research may extend the results of this paper to a more general setting.

The next result is a well-known method for theoretically calculating symplectic eigenvalues that suits our purposes. 
\begin{prop}
\label{sympev}
Let $N$ be an $n\times n$ real positive definite matrix with $n=2p$.
Then 
\ben[$(a)$]
    \item 
    \label{sympev-a}
    Each eigenvalue of $\Omega N$ is purely imaginary and nonzero.
    \item
     \label{sympev-b}
    The moduli of the eigenvalues of $\Omega N$ are the symplectic eigenvalues of $N$.
    \item  
    \label{sympev-c}
    The multiplicity of $\lambda$ as a symplectic eigenvalue of $N$ is half the number 
     of eigenvalues of $\Omega N$ having moduli  $\lambda$. This also equals the dimension of the nullspace of 
     $\Omega N-\lambda \sqrt{-1} I$.
\een 
\end{prop}
\begin{proof}
By Williamson's Theorem, there exists an $n\times n$ symplectic matrix $S$
such that $S\trans N S= D \oplus D$, where $D=\diag(d_1, \ldots,d_p)$  whose diagonal entries are the 
symplectic eigenvalues of $N$. Since $S$ is symplectic, 
so is $S\trans$.  Thus, $S \Omega S\trans=\Omega$ and  $\Omega N= S\Omega(D\oplus D) S^{-1}$. Hence, $\Omega N$
is similar to 
$
\Omega (D\oplus D)= \mtx {\begin{array}{rc} O & D \\  -D & O 
\end{array} }$,
whose eigenvalues are $\pm d_j\sqrt{-1} $ for $j=1,\ldots, p$.
Statements \eqref{sympev-a}--\eqref{sympev-c} readily follow.  
\end{proof}

We use the notation $\spec(A)$ to denote the (standard) spectrum of a square matrix $A$.

 \begin{ex}
\label{ex:directsum}
Let $A$ and $B$ be $p\times p$ positive definite matrices, and 
$N=A\oplus B$.   
Since $(\Omega N)^2=\mtx{-BA & O\\O & -AB}$ and $\spec(AB)=\spec(BA)$, the 
(standard) eigenvalues of $\Omega N$ are the square roots of the (standard) eigenvalues of $-AB$.
If we take $B=A^{-1}$, then $AB=I$ and the symplectic eigenvalues of $A \oplus B$ are all equal to one. If we take $B=A$, then $AB=A^2$ and the symplectic eigenvalues of $A \oplus A$ are the (standard) eigenvalues of $A$. 
\end{ex}

The set $\{Q\trans Q\colon Q \text{ is orthogonal}\}$ is equal to $\{I\}$. The analogous set in the symplectic setting  $\{S\trans S\colon S \text{ is symplectic}\}$ is richer, in terms of both matrices and graphs.  Williamson's Theorem implies that 
every matrix of the form $S\trans S$ with $S$  symplectic has each of its symplectic eigenvalues equal to one,  
while the identity matrix
$I$ is the  
only $n\times n$ symmetric matrix having each eigenvalue equal to one.
For the orthogonal case the corresponding graph is just the empty graph. However, for the symplectic case there are many non-trivial graphs. 
In particular, one is the only symplectic eigenvalue of the $n \times n$ matrix 
\begin{align}
\label{spec_arb_ex}
\mtx{ I & J \\
      O & I
}^\top \mtx{ I & J \\
     O & I}
      = \mtx{I & J \\ J &  pJ+I},
\end{align}
 where $J$ is the all ones matrix (see \Cref{pK1veeKp} for further discussion).  In \Cref{s:PD}, we list  several equivalent characterizations of $n\times n$ positive definite matrices that have a single eigenvalue (of multiplicity $p$), provide several other  examples of such matrices, and establish  a sharp lower bound  on the sparsity of such matrices when they are irreducible. 

We conclude this subsection with a  brief history 
and some recent uses of symplectic eigenvalues.   What we now call Williamson's Theorem  is actually a simple instance of Williamson's  general theory for canonical forms for matrix pencils. 
Arnol'd brought the concept of symplectic eigenvalues 
out of a matrix algebra setting and integrated it into differential geometry and Hamiltonian mechanics in \cite{A78}.

Two developments in the late 20th century brought new attention to symplectic eigenvalues. 
The first  was in the area of  symplectic topology, where symplectic eigenvalues 
were shown to be geometric invariants.
It is not difficult to see that if $N$ is a $2p\times 2p$ positive definite matrix, then the standard ellipsoid of smallest volume in which the set
$\{x\in \R^{2p}: x\trans Nx \leq 1\}$ can be orthosymplectically embedded is one having axes of lengths $1/\sqrt{d_i}$ $(i=1,\ldots, p)$. 
 This is a key observation in Gromov's non-squeezing theorem \cite{G84}:  a symplectic embedding of a ball $B^{2p}(r)$ into a cylinder $Z^{2p}(R) = B^2(R) \times \mathbb{R}^{2p-2}$ is possible only if the radius $r$ of the ball  is less than or equal to the radius $R$ of the cylinder's base. 
Gromov's theorem was then shown to be  deeply related to the Heisenberg uncertainty principle and  placed restrictions on how precisely position and momentum coordinates can be simultaneously squeezed or controlled in Hamiltonian systems \cite{G10}. 
The second development arose in quantum optics and quantum information, where the  covariance matrix  of a Gaussian state $S$  of $p$ bosonic modes is a positive definite $2p\times 2p$ matrix $N$, and the  symplectic eigenvalues of $N$  were shown to govern much of the physics of the system, e.g.,  the von Neumann entropy of 
the state $S$  is the sum of the symplectic eigenvalues of $N$. See \cite{HSH99} for more details. 

In recent years, symplectic eigenvalues have become a subject in matrix analysis in their own right. 
The 2015 paper by Bhatia and Jain \cite{BJ15} developed analogs of the classical eigenvalue theory  including: variational characterizations,  an analog of Ky Fan's principle, interlacing theorems, Weyl- and Lidskii-type inequalities and majorization results, perturbation bounds, and  results on the differentiability and analyticity of the symplectic eigenvalues as functions of $N$;  all this work is done for real positive definite matrices. 
Sets of symplectic matrices that simultaneously possess additional algebraic or numerical structures, 
including the positive definite symplectic matrices (which we will further study in  \Cref{s:PD}) are developed in \cite{Dopico-Johnson}. 
Numerical algorithms for calculating the symplectic eigenvalues are given in \cite{SABT21}. 
 The \textit{symplectic rank} of a $2p\times 2p$ positive definite matrix $N$  
is defined in \cite{WPPC} as the difference between  $p$ and  the multiplicity of the smallest symplectic eigenvalue of $N$. This is analogous to the well-studied  notion of minimum positive definite rank in the IEP-$G$ setting.

\subsection{Preliminaries} 
\label{ss:prelim}
In the setting of spectra of real symmetric matrices, there is a rich body of work on the relationship 
between the spectrum  and the associated graph of the matrix. In  this work 
we initiate the study of the analogous 
relationship in the symplectic setting. The next definitions allow us to make this more precise. 
 These will seem unduly formal to those familiar with the IEP-$G$, since as we will see, the  labeling of the graph does not affect the (standard) eigenvalues.  However, the corresponding symplectic problem is more subtle: while many re-labelings preserve the symplectic eigenvalues, not all do. 
 
 Throughout this paper,  the term \emph{graph} means an unlabeled, simple, undirected, and finite graph (with  a nonempty vertex set), whereas the term \emph{digraph} means an unlabeled directed graph.  
 That is, a graph of order $n$ is a pair $G=(V(G),E(G))$ where $V(G)=\{v_1,\dots,v_n\}$ is the \emph{vertex set}  and the \emph{edge set} $E(G)$  is a set of two element subsets of vertices; the edge $\{v_i,v_j\}$ is often denoted by $v_iv_j$ (see \Cref{ss:notation} for the formal definition of a digraph and additional graph definitions).
 A \emph{labeling} $L$ is a bijection $L:V(G)\to\{1,\dotsc,n\}$; this labeled graph is denoted by $G^L$.

 For such a labeled graph $G^L$,   $\sym(\GLG)$   denotes the set of all symmetric  matrices $A=[a_{ij}]$ 
 where for $i\ne j$, $a_{ij}\neq 0$ if and only if   $ij$ is an edge of $G^L$; in the literature, the labeling is suppressed and this is denoted by $\sym(G)$.  
 For a permutation $\sigma$ of $1,\ldots, n$,  the labeled graph obtained from $G^L$ by re-labeling 
the vertex labeled $i$ by the label $\sigma(i)$ is  denoted by $\GLGsig$.  Then the  matrices $A$ and $P_\sigma AP_\sigma^\top$ are permutationally similar, and 
$A \in  \sym(\GLG)$ if and only if $P{_\sigma} A P_\sigma^\top \in  \sym(\GLGsig)$.
 Thus, the inverse eigenvalue problem for $\sym(\GLG)$ and $\sym(\GLGsig)$
are the same. That is, for  the IEP-$G$,  the labeling of the vertices of $G$ is immaterial, so it is reasonable to suppress the labeling in the notation. 

For symplectic eigenvalues, there is more subtlety.  Unlike the (standard) eigenvalue case, there are constraints on the permutation $\sigma$ that are needed to guarantee the symplectic eigenvalues are preserved. 
That is, not all permutation similarities preserve symplectic eigenvalues. 
\begin{ex}
 Consider
$ 
N= \left[\begin{array}{rrrr}
2 & 0 & 1 & 1 \\
0 & 2 & 1 & 0 \\
1 & 1 & 2 & 0 \\
1 & 0 & 0 & 2
\end{array}\right]$ 
and $\sigma=(1\, 2\, 4)$.
Then 
$ 
P_{\sigma} = \left[ \begin{array}{cccc}
0 & 0 & 0 & 1\\
1& 0 & 0 & 0\\
0 & 0 & 1 & 0\\
0& 1 & 0 & 0 
\end{array} 
\right]$,
and 
 $P_\sigma NP_\sigma^\top=
\left[\begin{array}{rrrr}
2 & 1 & 0 & 0 \\
1 & 2 & 1 & 0 \\
0 & 1 & 2 & 1 \\
0 & 0 & 1 & 2
\end{array}\right]$.
The (approximate)  symplectic spectra  of $N$ and $ P_\sigma NP_\sigma^\top$ are  $\{ 1.902, 1.176\}$  and $\{3.078, {0.727}\}$,
respectively  (recall these can be computed as in Proposition \ref{sympev}).  
 This example,  for which the graphs are illustrated in Figure \ref{fig:1}, shows that the symplectic eigenvalues for these two labelings of the path on four vertices  need to be handled separately.

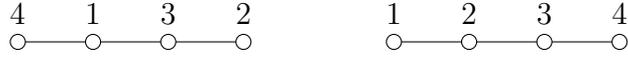
\begin{figure}[H]
\begin{center}
\begin{tikzpicture}
    \draw(1,0)--(4,0);\foreach\i in{1,2,3,4}{\draw[fill=white](\i,0)circle(0.1);};
    \node[above]at(1,0.1){$4$};\node[above]at(2,0.1){$1$};\node[above]at(3,0.1){$3$};\node[above]at(4,0.1){$2$};
    \draw(6,0)--(9,0);\foreach\i in{1,2,3,4}{\draw[fill=white](\i+5,0)circle(0.1);};
    \node[above]at(6,0.1){$1$};\node[above]at(7,0.1){$2$};\node[above]at(8,0.1){$3$};\node[above]at(9,0.1){$4$};
 \end{tikzpicture}
  \caption{Two labelings $L$ and $\sigma(L)$ of $P_4$ where the matrices $N\in \symp(P_4^L)$ and $P_\sigma NP_{\sigma}\trans\in \symp(P_4^{\sigma(L)})$ have different symplectic eigenvalues. 
 \label{fig:1}}\vspace{-10pt}
 \end{center}
\end{figure}
\end{ex}
Thus, it is imperative, when discussing 
symplectic eigenvalues of matrices associated with a labeled graph $G^L$, that the vertex labels of $G^L$ are  taken into consideration. 
  Given a labeled graph $G^L$, we will use $\symp(\GLG)$ to denote the set of  positive definite matrices  $N=[n_{ij}]$   where $n_{ij}\neq 0$ if and only if  $i=j$ or  $ij$ is an edge of $G^L$ (note that when studying the IEP-$G$,  $\symp$ is commonly used to denote the analogous set of positive semidefinite matrices).

Given an $n\x n$ positive definite matrix $N$, it is  useful to talk about the \emph{labeled graph of $N=[n_{ij}]$}, denoted by $\GL L(N)$, which has vertices $1,\dots,n$ and the edge set $\{ij\colon n_{ij}\neq 0$ and  $i\ne j\}$.

\subsection{Couplings}\label{ss:couplings}

It is the case that some labelings  result in the same symplectic eigenvalues because they can be obtained from one another by symplectic permutation matrices or other symplectic matrices 
 that preserve the unlabeled graph; we call such labelings \emph{symplectically equivalent}.
   When determining all symplectic eigenvalues of all labelings of a   graph $G$, it is desirable to consider only one representative from each class of symplectically equivalent labelings of $G$. 

A \textit{monomial matrix} is  a $(0,1,-1)$-matrix $R$ with exactly one nonzero entry in each row and column. 
Each monomial matrix has the form $R=EP$, where $E$ is a diagonal matrix whose diagonal entries are in $\{\pm 1\}$, 
and $P$ is a permutation matrix; note that $R\trans=R^{-1}$ for  each monomial matrix $R$. Let $R$ be a monomial matrix for which $R^{\top}{\Omega}R=\Omega$, and let $N$ be a positive definite matrix.
By \Cref{sympev}, the  symplectic eigenvalues of $N$ 
are the moduli of the eigenvalues of $\Omega N$.
Note that $\Omega N= R^{\top} \Omega R N$ is similar to $\Omega R N R^{\top} $, so $N$ and $RNR^{\top}$ have the same symplectic eigenvalues.   If  $R=EP_\sigma$ and the labeled graph of $N$ is $G^{L_0}$, then $ \G^L(RNR^{\top})=\G^L(P_{\sigma} NP_{\sigma}^{\top})=G^{\sigma({L_0})}$. 
Therefore,  for such $\sigma$, 
studying the symplectic eigenvalues of matrices with labeled graph $G^{L_0}$ is  the same as studying the symplectic eigenvalues of matrices with labeled graph $G^{\sigma({L_0})}$. 
 If $P$ is a permutation matrix, then  there exists a diagonal matrix $E$ such that $R=EP$ and $R\trans \Omega R= \Omega$ if and only if

\begin{equation}
\label{eq:perm} 
P\trans\begin{bmatrix} O & I \\ 
I & O 
\end{bmatrix} P= 
\begin{bmatrix} O & I \\ 
I & O 
\end{bmatrix}.
\end{equation} 
 Thus, to identify  graphs  $G^{\sigma({L_0})}$ that allow the same symplectic spectra as $G^{L_0}$,  we need to find  the permutation matrices $P$ such that 
\[ 
\left[
\begin{array}{cc}
O & I\\
I & O  
\end{array}
\right]
\left[ \begin{array}{cc} 
 P_{11} &  P_{12} \\
 P_{21} & P_{22} 
 \end{array} \right] 
  \left[
 \begin{array}{cc}
O & I\\
I & O  
\end{array}
\right]
 = 
\left[ \begin{array}{cc} 
 P_{11} &  P_{12} \\
 P_{21} & P_{22} 
 \end{array} \right]. 
\]
This is equivalent to $P_{21}=P_{12}$  and $P_{11}=P_{22}$. 
If the permutation corresponding to $P$ is $\sigma$, one can check that this implies 
$\sigma(i) = j$  if and only if $\sigma(i+p)=j+p$ (reading indices mod $2p$).  So the permutations $\sigma$ that work are those such that 
\[ \{ \{ \sigma(1), \sigma(1+p)\}, \ldots, \{ \sigma(p), \sigma(2p)\} \}= \{ \{ 1, 1+p\}, \ldots, \{p, 2p\} \}. \] 
So in particular, there are $p! 2^p$ valid re-labelings. 

We can describe symplectically equivalent labeled graphs by couplings using vertex names.  

\begin{defn}
 A \textit{coupling} of $\{v_1,\ldots, v_{2p}\}$ is a set $\PP=\{(c_1,d_1), \ldots, (c_p,d_p)\}$ such that each  of the vertices appears exactly once in $\PP$.
A \emph{coupled graph}, denoted by $\GC G C$, is a graph $G$ of order $n=2p$     with vertex set $V=\{v_1, v_2, \ldots, v_n\}$ and a coupling  $\PP$.  For each vertex $v_i$ in a coupled graph, we denote by $\pp(v_i)$ the vertex paired with $v_i$.
\end{defn}

Although we use the notation $(\cdot,\cdot)$ in the definition of a coupling to avoid confusion with graph edges, the pairs are unordered 
 in the coupling $\PP$.  To study the Inverse Symplectic Eigenvalue Problem of a coupled graph $\GC G C$, we  choose  a representative labeling.   
 The family of all real positive definite matrices associated with the $2^pp!$ labelings described by a coupled $\GC G C$  is denoted by $\symp(\GC G C)$.

To see that there are  $(2p-1)(2p-3)\cdots (3)(1)$ couplings of a graph $G$ of order $n=2p$, note that there are $2p-1$ choices for   $\pp(v_1)$,   there are $2p-3$ choices for $\pp(v_j)$ where $j$ is  the  lowest index  excluding vertices $v_1$ and $\pp(v_1)$, etc.  
 This is illustrated in the next remark.

\begin{rem}\label{r:ord4couplings}
     When studying  
the  sets of symplectic eigenvalues allowed by 
matrices whose (unlabeled) graph $G$ has vertex names $v_1,v_2,v_3,v_4$,  there are  three couplings, 
 $\PP_1=\{(v_1,v_2)$, $(v_3,v_4)\}$, $\PP_2=\{(v_1,v_3),(v_2,v_4)\}$, $\PP_3=\{(v_1,v_4),(v_2,v_3)\}$.  When determining possible symplectic eigenvalues of matrices in $\symp(G^\PP)$, we  therefore need to consider at most three distinct representative labelings, one for each coupling (symmetries of the unlabeled graph sometimes reduce the number of couplings that need to be considered).
\end{rem}

By the discussion above, if $G^{L_1}$ and $G^{L_2}$ are two labelings of a graph $G$  described by one coupling, then $\symp(G^{L_1})$ and $\symp(G^{L_2})$ have the same symplectic spectrum.  Associated to the coupled graph $G^\PP$  
with coupling $\PP=\{(c_1,d_1), \ldots, (c_p,d_p)\}$ is the set of $2^pp!$ labeled graphs obtained by  choosing a permutation $\sigma$ of $1,\dots,p$ and assigning the labels  $\sigma(k)$ and  $\sigma(k)+p$    
to $c_k$ and $d_k$ (either order) for $k=1,\dots,p$; each of these is called a \emph{representative labeling}.  

 \begin{ex}\label{e:ord4couplings}
 There are 
eight  representative labelings of a  
coupled graph $\GC G C$, where the vertices are $v_1,v_2,v_3,v_4$ and the coupling is $\PP=\{(v_1,v_3), (v_2,v_4)\}$. Listed in lexicographical order they are:

\[
v_1\mapsto 1, v_2\mapsto 2, v_3\mapsto 3, v_4\mapsto4; \qquad 
v_1\mapsto1, v_2\mapsto4, v_3\mapsto3, v_4\mapsto2; \]
\[v_1\mapsto2, v_2\mapsto1, v_3\mapsto4, v_4\mapsto3; \qquad 
v_1\mapsto2, v_2\mapsto3, v_3\mapsto4, v_4\mapsto1;  \]
\[v_1\mapsto3, v_2\mapsto2, v_3\mapsto1, v_4\mapsto4; \qquad 
v_1\mapsto3, v_2\mapsto4, v_3\mapsto1, v_4\mapsto2; \] 
\[v_1\mapsto4, v_2\mapsto1, v_3\mapsto2, v_4\mapsto3; \qquad 
v_1\mapsto4, v_2\mapsto3, v_3\mapsto2, v_4\mapsto1.
\] 
\end{ex}

 Given an $2p\x 2p$ positive definite matrix $N$,  the \emph{coupled graph of $N=[n_{ij}]$}, denoted by $ \G^\II(N)$, is the coupled graph that has $\G^L(N)$ as a representative labeled graph.

\subsection{Questions investigated\label{ss:contrib2}}
We are now able to begin studying  spectral problems related to  $\symp(\GC G C)$ where $\GC G C$ is a coupled graph on $n=2p$ vertices, such as:
\ben[(a)]
\item\label{item-a} Which multi-sets of $p$ positive real numbers are the symplectic spectra of  {matrices} in  $\symp(\GC G C)$?
 \item\label{item-d} Which {ordered} partitions  of $p$ can be realized as  the {ordered} lists of  multiplicities of symplectic eigenvalues of {matrices} in  $\symp(\GC G C)$?
\item\label{item-b} What is the largest multiplicity of a symplectic eigenvalue among the matrices in  $\symp(\GC G C)$?
\item\label{item-dx}  
Which coupled graphs allow a positive definite matrix having only one symplectic eigenvalue?
\een

 Questions \eqref{item-a}, \eqref{item-d}, and \eqref{item-b} are the analogs of the  well-studied Inverse Eigenvalue Problem, Inverse Multiplicity Problem, and Maximum Multiplicity Problem for Graphs \cite{HLSbook}; 
we call these the \textit{Inverse Symplectic Eigenvalue Problem for Coupled Graphs} (the ISEP-$G$, for short), the \emph{Inverse Symplectic Multiplicity Problem for Coupled Graphs}, and the \emph{Maximum Symplectic  Multiplicity Problem for Coupled Graphs}. 
Question  \eqref{item-dx} can be thought of as the symplectic analog of the study of graphs that allow a symmetric matrix having only two distinct eigenvalues.  As noted in \cite{F+13}, such a graph allows an orthogonal symmetric matrix, and as shown in \Cref{s:PD}, a graph that allows a positive definite matrix having only one symplectic eigenvalue allows a symplectic positive definite matrix.
The rest of this paper  is devoted to studying various aspects of these problems.  

In  \Cref{s:order4} we   give a complete resolution of the ISEP-$G$ for all coupled graphs of order four by applying the
theory developed in 
\Cref{s:strong}  and developing additional tools for solving the ISEP-$G$.  
  There are many examples of coupled graphs $\GC G C$ on $n=2p$ vertices such that 
every multi-set of $p$ positive real numbers is the symplectic spectrum of some $N\in\symp(\GC G C)$. 
We call such  
labeled graphs \textit{symplectic spectrally arbitrary} and study such  graphs in \Cref{s:ssa}.

In \Cref{s:strong} 
we derive the  Strong Symplectic Spectral Property (SSSP, for short), state several useful theorems, and  illustrate how to use these theorems
to establish  results about the ISEP-$G$ { (some of the proofs that are similar to those for the IEP-$G$ are deferred to \Cref{appendix}). In particular, in  \Cref{s:strong} we resolve the  Inverse Symplectic Multiplicity Problem  for each coupled graph  $\GC G C$ such that there is a matrix in $\symp(\GC G C)$ that has only one symplectic eigenvalue and  has the SSSP.

In \Cref{s:zeroforcing}, we introduce a new type of zero forcing (coupled zero forcing) and show that  the coupled zero forcing number is an upper bound for the maximum symplectic multiplicity.  We also characterize all coupled graphs with coupled zero forcing number equal to one and solve the ISEP-$G$ for several additional families of coupled graphs.
We conclude in \Cref{s:conclude} with a summary of the coupled graphs for which we have solved the ISEP-$G$ and the tools used to solve them, and with questions for future research}.

\ms

{\subsection{Notation and Terminology}\label{ss:notation}

The remainder of this section contains some standard notation and terminology. We use $|\mu|$  to denote the modulus (also known as the magnitude or absolute value) of the complex number $\mu$. 

For an $m\times n$ matrix 
$A$, $\alpha \subseteq \{1,\ldots, m\}$ and $\beta\subseteq\{1,\ldots n\}$,  we define $A[\alpha, \beta]$ (respectively, $A(\alpha, \beta)$) to be the submatrix of $A$ obtained by keeping (respectively, deleting) 
those rows in $\alpha$ and those columns in $\beta$.
When $A$ is square, we abbreviate 
$A[\alpha,\alpha]$ to $A[\alpha]$, and $A(\alpha, \alpha)$ to $A(\alpha)$.
An $n\times n$ matrix $A$  is \textit{reducible} provided there is a permutation matrix $P$ such that 
$PAP^\top$ has the block matrix form 
\[
\begin{bmatrix}
A_{11} & A_{12}
\\
O & A_{22}
\end{bmatrix},
\]
where 
$A_{11}$ and $A_{22}$ are square, non-vacuous matrices. 
A matrix is \textit{irreducible} if it is not reducible  (in particular, each $1\times 1$ matrix is irreducible).

 A graph $G$ is a \emph{subgraph} of a graph $H$ if $V(G)\subseteq V(H)$ and $E(G)\subseteq E(H)$; in this case, we also say that $H$ is a \emph{supergraph} of  {$G$}.  For $U\subseteq V(G)$, the \emph{subgraph  of $G$ induced by $U$}, denoted by $G[U]$, is the subgraph of $G$ with $V(G[U])=U$ and $E(G[U])=\{uv\colon uv \in E(G) \mbox{ and } u,v\in U\}$; $G[U]$ is also called an 
\emph{induced subgraph} of $G$. 
The \emph{complement} of $G$, denoted by $\OL{G}$, is the graph having vertex set $V(G)$ and edge set that includes edge $uv$ if and only if $uv\not\in E(G)$. 
The \emph{union} of graphs $G$ and $G'$, denoted by  $G\,\cup\, G'$, is the graph with $V(G\cup G')=V(G)\cup V(G')$ and $E(G\cup G')=E(G)\cup E(G')$; if $V(G)\cap V(G')=\emptyset$, this is a \emph{disjoint union} of graphs and is denoted by $G\du G'$. If $V(G)$ and $V(G')$ are disjoint, the \emph{join} of $G$ and $G'$, denoted by $G\vee G'$,  is the graph obtained from $G\du  G'$ by adding all possible edges between $V(G)$ and $V(G')$.  
 
Vertices $u,v\in V(G)$ are \emph{adjacent}  if $uv\in E(G)$. 
The (open) \emph{neighborhood} of a vertex $v$ is $N_G(v)=\{u\in V(G)\colon uv\in E(G)\}$, and the \emph{closed neighborhood} is $N_G[v]=N_G(v)\cup\{v\}$; when $G$ is clear from the context, the subscript may be omitted.  The \emph{degree}  of a vertex $v\in V(G)$ is $\deg_G(v)=|N_G(v)|$. 
A vertex of degree one is called a \emph{leaf}. The minimum degree among the vertices of the graph $G$ is denoted by $\delta(G)$.   A \emph{matching} in a graph $G$ with $V(G)=\{v_1,\dots,v_n\}$ is a set  $\{v_{i_1}v_{j_1},\dots, v_{i_k}v_{j_k}\}\subseteq E(G)$ such that  all the vertices $v_{i_s}, v_{j_t}, $ for $s,t=1,\dots,k$ 
are distinct. A \emph{perfect matching} in  $G$ is a matching that includes all vertices of $G$.

 A \emph{digraph} $\Gamma=(V(\Gamma),E(\Gamma))$ is a simple directed graph. The  set of vertices $V(\Gamma)$ is finite  and nonempty. An \emph{arc} is an ordered pair of vertices and $E(\Gamma)$ denotes the set of arcs of $\Gamma$, so \emph{loops} (arcs of the form $(v,v)$) and both the arcs $(u,v)$ and $(v,u)$ are allowed, but duplicate arcs are not allowed.  
 A digraph $\Gamma$ is 
\textit{strongly connected}
provided that for 
each ordered pair of vertices $(u,v)$ of $\Gamma$ there is a directed walk in $\Gamma$ from $u$ to $v$. The adjacency matrix $\A_\Gamma$ of $A=[a_{ij}]\in\Rnn$ has $V(\Gamma)=\{1,\dots,n\}$ and $(i,j)\in E(\Gamma)$ if and only if $a_{ij}\ne 0$.
Thus, $\Gamma$ is strongly connected if and only if  the  adjacency matrix of $\Gamma$ is irreducible. The \textit{strong components} of $\Gamma$ are the maximal induced {strongly connected} subdigraphs of $\Gamma$. 
The vertex sets $V_1$, $\ldots, V_k$ of the strong components partition the vertices of $\Gamma$ and can be ordered so that if $(u,v)$ is an arc of $\Gamma$
with $u\in V_i$ and $v \in V_j$, then $i\leq j$.

Graphs $G$ and $H$ are 
isomorphic, denoted  by 
$G \cong H$,
provided 
there is a bijection $\phi$ between the vertex set of $G$ and the vertex set of $H$ such that $u$ and $v$ are adjacent in $G$ if and only if $\phi(u)$ and $\phi(v)$ are adjacent in $H$.  The notion of isomorphic digraphs is defined similarly. 
 

\section{Symplectic positive definite matrices}\label{s:PD}

We begin the study of the ISEP-$G$ with a problem whose analog for the IEP-$G$ has a very simple resolution,  but for the ISEP-$G$ is much more complex.  An $n\times n$ symmetric matrix  $A$ that has   only one eigenvalue is a scalar matrix, and hence its graph  $\G(A)\cong \OL{K_n}$ is disconnected if $n>1$. 
However, in the symplectic setting the situation is very different, as we saw that the matrix in  \Cref{spec_arb_ex} is both symplectic and positive definite, and has exactly one symplectic eigenvalue.

 If $N$ is positive definite and $\lambda $ is a positive real number, then each symplectic eigenvalue of $N$ is  equal to $\lambda$ if and only if $(1/\lambda )N$ is positive definite with each symplectic 
eigenvalue equal to one. Thus we focus on positive definite matrices having each symplectic eigenvalue equal to one. 
 
 We call a symplectic positive definite matrix 
a \emph{sympPD} matrix.
As was done in \Cref{s:intro}, we can use the form $S\trans S$ with $S$ symplectic  to construct examples of sympPD matrices with specific graphs.

 \begin{ex}\label{pK1veeKp}
 The matrix 
 in \eqref{spec_arb_ex}, which is a sympPD matrix, has  a coupled graph  
$(\overline{K_{p}} \vee K_{p})^\MM$
where {$V(\overline{K_p})=\{v_1,\ldots, v_p\}$
and $V(K_p)=\{ v_{p+1}, \ldots, v_{2p}\}$,  and  $\MM=\{(v_1,v_{p+1}),$ $(v_2,v_{p+2}),\dots,(v_p,v_{2p})\}$} (we {often} use $\MM$ to
denote a {specific} coupling such that the graph includes  the perfect matching with  edges $v_i\mm(v_i), i=1,\dots,p$; 
matchings will play an important role in \Cref{s:zeroforcing}).     
\end{ex}

\begin{ex}\label{KppSympPD} 
Let $B=[b_{ij}]$ be a symmetric   orthogonal $p\times p$ matrix,  $S=\mtx{I & B\\O & I}$, and  $N=S\trans S=\mtx{ I & B \\      B & 2I }$, so $N$ is a sympPD matrix.
  By choosing a symmetric orthogonal matrix $B$ with all nonzero entries, e.g., a Householder matrix created from a vector with all nonzero entries, 
we see that $N=\mtx{I & B \\
       B &  2I}\in \symp(K_{p,p}^\MM)$ where  $ K_{p,p}^\MM$ is the coupled graph with partite sets {$\{v_1,\ldots, v_p\}$
and $\{ v_{p+1}, \ldots, v_{2p}\}$ and  $\MM=\{(v_1,v_{p+1}),$ $(v_2,v_{p+2}),\dots,(v_p,v_{2p})\}$.} 
\end{ex}

 In  \Cref{ss:sympPD-char} we {summarize and apply}  known results about symplectic positive definite matrices to characterize these matrices, {noting that} this is equivalent to {characterizing matrices with} all symplectic eigenvalues equal.  In \Cref{ss:PDirred-sparse} we establish a lower bound on  the number of edges in a connected graph that allows all symplectic eigenvalues to be equal.  In \Cref{ss:TPn} we construct a family of 
 sparse graphs, which we call triangular paths, that allows  all symplectic eigenvalues to be equal and attains  
 the lower bound established in \Cref{ss:PDirred-sparse}. In \Cref{ss:forbid} we identify structures that prevent  a graph from allowing all symplectic eigenvalues to be equal.

\subsection{Characterization}\label{ss:sympPD-char}

The statements in the next remark are straightforward and well-known. We list them here for convenience.
\begin{rem}\label{r:sympPDequiv}
Let $N$ be a positive definite 
matrix.  The following statements 
are equivalent:
\bit
\item $N$ is symplectic, i.e., $N$ is a sympPD matrix.
\item Every symplectic eigenvalue of $N$ is one.
\item $N=S\trans S$ for some symplectic matrix $S$.
\item $(\Omega N)^2=-I$.
\item $N^{-1}= -\Omega N\Omega= \Omega N\Omega\trans$.
\eit
\end{rem}

\begin{prop}
\label{sympd}
Let 
$ 
N=
\renewcommand{\arraystretch}{1.1}
\mtx
{N_{11} & N_{12}\\ 
N_{12}\trans & N_{22} }$
be a positive definite matrix, 
where $N_{11}$ and $N_{22}$ are $p\times p$. 
Then $N$ is a sympPD matrix if and only if
\beq\label{eq:sympPDinv} N^{-1}= 
\left[ \begin{array}{rr} N_{22}& -N_{12}\trans \\ 
-N_{12} & N_{11} \end{array} \right]. \eeq
In this case, 
$\G(N^{-1})\cong \G(N)$.

\end{prop}

\bpf As noted in \Cref{r:sympPDequiv},
  $N$ being a sympPD matrix is equivalent to $N^{-1}= -\Omega N\Omega$, which is equivalent to $N^{-1}$ having the  form in \eqref{eq:sympPDinv}. 
 
 Now assume $N$ is a sympPD matrix, so $N^{-1}=\Omega N \Omega\trans$.  Since  $\Omega=EP_\tau$ where $E=I\oplus -I$ 
 and $\tau=(1\ 1\!+\!p)(2\ 2\!+\!p)\cdots(p\ 2p)$, we see that
$\G(N^{-1})\cong \G(N)$. 
\epf

Dopico and Johnson established the next theorem in 2009. It plays a central role in the study of sympPD matrices. \begin{thm}\label{t:DJ} {\rm \cite{Dopico-Johnson}} 
   The set of symplectic positive definite matrices is
\begin{equation}
\label{symppsd}
\renewcommand{\arraystretch}{1.1}
\lsb\left[ \begin{array}{cc} 
N_{11} & N_{11}W \\  
WN_{11} & N_{11}^{-1} + WN_{11}W
\end{array}
\right] 
 \colon N_{11} \mbox{ is positive definite and $W$ is symmetric}\rsb.  
\end{equation}
\end{thm}

 The complete graph allows a sympPD matrix {(see \Cref{random})}, as do many other dense graphs.  It is natural to ask how   sparse  a connected coupled graph $G^\PP$ can be if it allows a sympPD matrix  $N\in \symp(G^\PP)$.  The next  two sections address this question.


\subsection{Sparsity of inverse pairs of irreducible positive definite matrices}\label{ss:PDirred-sparse}

 The number of nonzero entries in a matrix $A$
is denoted by $\#(A)$.
 In this section, we establish a lower bound on $\#(M)+\#(M^{-1})$ for an irreducible positive definite matrix $M$.  Note that $M$ is irreducible if and only if $M^{-1}$ is irreducible.
\Cref{sympd} then allows us to translate this to a lower bound for $\#(N)$ when $N$  is an irreducible positive definite matrix 
(see \Cref{c:sympPD-sparseLB}). In the next section we show that this bound is sharp (see  \Cref{tp_is_sympd}). 

A symmetric  matrix $A$    is irreducible if and only if $\G(A)$ is connected.
 In the proof of the next theorem we use unlabeled graphs 
because irreducibility is invariant under permutation similarity,  so the specific labeling of the matrix is irrelevant and is therefore suppressed.

\begin{thm}\label{t:PDirred-sparseLB}
Let $M$ be an $n\times n$ irreducible
positive definite matrix. 
Then 
\[\#(M)+ \#(M^{-1}) \geq 8n-8.\]
\end{thm}

\begin{proof}
The proof is by induction on $n$.
{For $n=1$, $\#(M)+ \#(M^{-1})= 1+1 > 0=8\cdot 1 -8$.} 
For $n=2$, the irreducibility of $M$ implies 
that $\#(M) + \#(M^{-1})=4 +4 =8\cdot 2 -8$.

Now assume that $n>2$ and proceed by induction.
Throughout {this proof} we use $N$ to denote $M^{-1}$.  Since $M$ is positive definite, every diagonal entry is nonzero, so the number of nonzero vertices in row $i$ of $M$ is $1+\deg_{\G(M)}(i)$, and similarly for $N$.
If each vertex of $\G(M)$ and $\G(N)$ has degree at least three, then 
\[ 
\#(M)+ \#(N) \geq 4n+4n=8n> 8n-8,
\]
as desired. Thus, we may assume that $\G(M)$ or $\G(N)$ 
has a vertex of degree at most two.
We consider several cases. 

\bigskip\noindent
{\bf Case 1.} $\G(M)$ or $\G(N)$
has a vertex of degree one.

Without loss of generality, we can assume that  vertex
$1$ is a vertex of degree one and it is adjacent to  vertex $2$ in $\G(M)$.
Thus $M$ has the form 
$
\begin{bmatrix}
m_{11} & m_{12} & 0 &\cdots & 0\\
m_{12} & m_{22} \\
0\\
\vdots\\
0 
\end{bmatrix}$.
 By considering the first row of the product  $MN=I$, we see that in the graph $\G(N)$ of inverse matrix $N$ we have an equality of neighborhoods: {$N_{\G(N)}[1]=N_{\G(N)}[2]$}. Since $\G(M)$, and hence $\G(N)$,
is connected and $n\geq 3$, vertex $1$ of $\G(N)$ has degree at least two. 

Let 
$P= \begin{bmatrix}
    1 & -m_{12}/m_{11}\\
    0 & 1
\end{bmatrix} \oplus I_{n-2}$,
$\widehat{M}= P\trans MP$,  and $
\widehat{N}= P^{-1} N (P^{-1})\trans$.
Then 
$\widehat{M}$ and $\widehat{N}$ are
 positive definite matrices that are inverses of each other.  Examination of $\widehat{M}=P\trans MP$ shows that the  matrix $\widehat{M}$ has the form 
$[m_{11}] \oplus B$  where (with the rows and columns of $B$ indexed by $2,\dots,n$)
$b_{22}=m_{22}-\frac{m_{12}^2}{m_{11}}$ and $b_{ij}=(M(1))_{ij}$ for $i,j\ge 2$ and $(i,j)\ne (2,2)$. Hence $\#(B)= \#(M(1))$.  Examination of $\widehat{N}=P^{-1} N (P^{-1})\trans$ shows that $\widehat{N}=
[1/m_{11}] \oplus N(1)$. 

Since $\G(M)$ is connected and vertex $1$ is  adjacent only to vertex 2, the graph $\G(M(1))$ is connected.
Hence $B$ is irreducible symmetric positive definite of order $(n-1)$ with inverse $N(1)$.
By induction, $\#(B) + \#(N(1))\geq 8(n-1)-8=8n-16$, implying that $\#(M(1))+ \#(N(1)) \geq 8n-16$.
The number of nonzeros in the first row and column of $M$ is three, the number of nonzeros in the first row and column of $N$ is at least five.
Therefore $\#(M)+ \#(N)\geq 8n-16 + 3 +5= 8n-8$,
as desired.

\bigskip\noindent 
{\bf Case 2.} All vertices of $\G(M)$ and $\G(N)$ have degree at least two and at least one has degree two. 

Without loss of generality, we may assume that vertex 1 has degree two and the neighbors of vertex $1$
in $\G(M)$ are vertices $2$ and $3$.  Thus $M$ has the form 
\[
M= \begin{bmatrix}
    m_{11} & m_{12} & m_{13} & 0 & \cdots & 0 \\
    m_{12}\\
    m_{13}
    \\
    0 \\
    \vdots\\
    0
\end{bmatrix}.
\]
Let 
$
P= \begin{bmatrix}
    1 & -m_{12}/m_{11} & -m_{13}/m_{11}\\
    0 & 1 & 0 \\
    0 & 0 & 1
\end{bmatrix} \oplus I_{n-3}$,
$\widehat{M}= P\trans M P$, and
$\widehat{N}= P^{-1} N(P^{-1})\trans$.
Then $\widehat{M}$ and $\widehat{N}$ are symmetric positive definite matrices that are inverses of each other. Examination of $\widehat{M}=P\trans MP$ shows that  $\widehat{M}= [m_{11}] \oplus B$  where $b_{22}=-\frac{m_{12}^2}{m_{11}} + m_{22}$, $b_{23}=-\frac{m_{12}m_{13}}{m_{11}} + m_{23}$, $b_{33}=-\frac{m_{13}^2}{m_{11}} + m_{33}$, and $b_{ij}=m_{ij}$ for all other entries of $B$.
 Thus $\G(B)$ is  $\G(M(1))$ or is obtained from $\G(M(1))$ 
by deleting  or inserting  the edge joining vertices $2$ and $3$. Examination of $\widehat{N}=P^{-1} N (P^{-1})\trans$ shows that $\widehat{N}=
[1/m_{11}] \oplus N(1)$.  Since $B$ and $N(1)$ are inverses of each other, either both   $\G(B)$ and $\G(N(1))$ are connected or  both $\G(B)$ and $\G(N(1))$ are  disconnected.   

\bigskip\noindent
{\bf Subcase:} $\G(N(1))$ is connected.
\\
In this case, both  $\G(B)$  and $\G(N(1))$ are connected. So $B$ and $N(1)$ 
are irreducible, positive definite symmetric matrices of order $n-1$ 
and are inverses of each other. By induction, \vspace{-3pt}
\[ \#(B) + \#(N(1))\geq 8(n-1)-8=8n-16.\vspace{-3pt}\]
The first row and column of $N$ contain at least five
nonzeros since all vertices of $\G(N)$ have degree at least 
two.
So $\#(N)\geq \#(N(1))+5$.  The first row and column of $M$ contain  at least five
nonzeros, and $\#(M(1))\geq \#(B)- 2$. 
So, 
$\#(M) \geq \#(B)+3$.  It now follows that \vspace{-3pt}
\[ \#(M) +\#(N) \geq 8n-16 + 5 + 3=8n-8,\vspace{-3pt}\] as desired. 

\bigskip\noindent
{\bf Subcase:} $\G(N(1))$ is disconnected.
\\
Since $B$ and $N(1)$ are inverses of each other, $B$ is also disconnected.  
Since $\G(B)$ and $\G(M(1))$ have the same edges except for possibly the edge joining $2$ and $3$, $\mathcal{G}(N(1))$ has exactly two 
connected components, one containing $2$ and 
the other containing $3$. Denote the vertex sets of the two connected components of $N(1)$ by $ \alpha$ and $ \beta$ where $2 \in \alpha$
and $3 \in \beta$.
Then without loss of generality, $M$ has the form 
\[ 
M= 
\left[ \begin{array}{ccc|c|c}
m_{11} & m_{12} & m_{13} & \mathbf 0^{\top} & \mathbf{0}^{\top}\\
m_{12} & m_{22} & m_{23} & \mathbf g^{\top} & \mathbf h^{\top} \\
m_{13} & m_{23} & m_{33} & \mathbf i^{\top} & \mathbf j^{\top}  \\ \hline
\mathbf 0 & \mathbf g &\mathbf i &  M[\alpha\setminus\{2\} ] & O 
\\ \hline 
\mathbf 0 & \mathbf h & \mathbf j &  O & M[\beta\setminus \{3\}] 
\end{array} \right] .  
\] 
and $N$ has the form 
\[
N= 
\left[ 
\begin{array}{ccc|c|c} 
n_{11} & n_{12} & n_{13} & \mathbf u^{\top} & \mathbf v^{\top} \\
n_{12} & n_{22}  & 0 & \mathbf x^{\top} & \mathbf 0^{\top}\\
n_{13} & 0 & n_{33} &  \mathbf 0^{\top} & \mathbf y^{\top} \\ \hline 
\mathbf u & \mathbf x &  \mathbf 0 &  N[\alpha\setminus\{2\}] & O \\ \hline 
\mathbf v & \mathbf {0} & {\mathbf y} & O & N[\beta
\setminus\{3\} ]
\end{array} 
\right] .
\] 
Note that  
\[ 
P^{\top}MP =
\left[ \begin{array}{ccc|c|c}
m_{11} & 0 & 0 & \mathbf 0^{\top} & \mathbf{0}^{\top}\\
0 & m_{22}-m_{12}^2/m_{11} & m_{23}-m_{12}m_{13}/m_{11} & \mathbf g^{\top} & \mathbf h^{\top} \\
0 & m_{23}-m_{12}m_{13}/m_{11} & m_{33}-m_{13}^2/m_{11}  & \mathbf i^{\top} & \mathbf j^{\top}  \\ \hline
\mathbf 0 & \mathbf g &\mathbf i &  M[\alpha\setminus \{2\} ] & O 
\\ \hline 
\mathbf 0 & \mathbf h & \mathbf j &  O & M[\beta\setminus \{ 3 \}] 
\end{array} \right] ,  
\] 
and 
$ P^{-1}N(P^{-1})^{\top}$
equals 
\begin{center}
\resizebox{.95\linewidth}{!}{$\left[ 
\begin{array}{ccc|c|c} 
 1/m_{11} & n_{12} + n_{22}m_{12}/m_{11} & n_{13} + n_{33}m_{13}/m_{11} & \mathbf u^{\top}+ (m_{12}/m_{11})\mathbf x^{\top}  & \mathbf v^{\top} + (m_{13}/m_{11}) \mathbf y^{\top} \\
n_{12}+n_{22}m_{12}/m_{11}  & n_{22} & 0 & \mathbf x^{\top} & \mathbf 0^{\top}\\
n_{13} +n_{33}m_{13}/m_{11} & 0 & n_{33} &  \mathbf 0^{\top} & \mathbf y^{\top} \\ \hline 
\mathbf u+ (m_{12}/m_{11})\mathbf x & \mathbf  x &  \mathbf 0 &  N[\alpha \setminus \{ 2\} ] & O \\ \hline 
\mathbf v + (m_{13}/m_{11})
\mathbf y & \mathbf {0} & \mathbf {y} & O & N[\beta\setminus \{3\} ]
\end{array} 
\right]$}.
\end{center}
Since $P^{\top} MP[\alpha , \beta]=O$, $m_{23}=m_{12}m_{13}/m_{11}$.  
Thus $m_{23} \neq 0$, that is,  there is an edge joining vertex $2$ and vertex $3$ in $\G(M)$. 
Since $P^{\top}MP$ and $P^{-1}N(P^{-1})^{\top}$ 
are inverses of each other, each off-diagonal entry in the first row of 
$P^{-1} N(P^{-1})^\top$ 
is zero. 
Thus, 
\begin{align*}
n_{12}& =-n_{22}m_{12}/m_{11},\\
n_{13}&=-n_{33}m_{13}/m_{11},\\
\mathbf u& = -(m_{12}/m_{11}) \mathbf x, \mbox{ and}\\ 
\mathbf v&= -(m_{13}/m_{11}) \mathbf y. 
\end{align*}
In particular, this implies that for $j \in \alpha \setminus \{2\} $,  vertex  $j$ is a neighbor of vertex $2$ in $\G(N)$ if 
and only if $j$ is a neighbor of $1$ in $\G(N)$.
Since each vertex of $\G(N)$ has degree at least  two, there is some $j \in \alpha \setminus \{2\} $ that  is a neighbor of vertex $2$. 
Hence vertex $1$ has at least two neighbors in $\alpha$. Similarly, vertex $1$ has at least two neighbors in $\beta$, and the degree $d$ 
of vertex $1$ in 
$\G(N)$ is at least four.

Up to permutation similarity, 
$B= B[\alpha]  \oplus B[\beta ]$
and $N(1)= N[\alpha] \oplus N[\beta]$. So $B[\alpha]$ and $B[\beta]$ are irreducible positive definite matrices of order $|\alpha|$ and $|\beta|$ respectively, and their inverses are $N[\alpha]$ and $N[\beta]$ respectively. 
Hence, by induction $\#(B[\alpha]) + \#(N[\alpha])\geq 8|\alpha| - 8$, and 
$\#(B[\beta]) + \#(N[\beta])\geq 8|\beta|-8$.

Note that $N$ has $1+ 2 {d}$ nonzero entries in its first row and column,  $M$ has five 
nonzero entries in its first row and column, and $\G(B)$ is $\G(M(1))$ with exactly one edge removed.  
So $\#(M(1))=
\#(B)+ 2$. Hence, $\#(M) =\#(B[\alpha]) + \#(B[\beta])+7$.
It follows that 
\begin{align*} 
\#(M) +\#(N)  
&\geq \#(B[\alpha]) + \#(B[\beta]) +7 + \#(N[\alpha]) + \#(N[\beta]) + (1+ 2d)
\\
&= \#(B[\alpha]) + \#(N[\alpha])+ \#(B[\beta]) + \#(N[\beta]) +8 + 2d \\
&\geq 8|\alpha|-8+ 8|\beta|-8 + 8+ 2d \\
&=8(n-1) + 2d -8 
\\
&\geq 8n-8. \qedhere 
\end{align*}
\end{proof}

\begin{cor}\label{c:sympPD-sparseLB} Let $n=2p$. 
Let $N$ be an  $n \times n$ sympPD matrix.  If $N$ is irreducible, i.e., $\G(N)$ is connected, then 
\[
\#(N) \geq 4n-4
\]
and $\G(N)$ has at least $3p-2$ edges. 
\end{cor}

\begin{proof}
Assume $N$ irreducible.  Since $N$ is sympPD, $N^{-1}= \mtx{ N_{22}& -N_{12}\trans \\ 
-N_{12} & N_{11} }$ by \Cref{sympd}.
Thus by  \Cref{t:PDirred-sparseLB},
$ 
2\#(N) = \#(N) + \#(N^{-1}) \geq 8n-8
$
and  the desired inequality follows. 

Since every diagonal entry of $N$ is positive and every edge is associated with two nonzero entries in the matrix, the number of edges in 
  $\G(N)$ is $\frac 1 2(\#(N)-n)\ge 3p-2$ edges.
\end{proof}

In order to show that the bounds in \Cref{c:sympPD-sparseLB} are sharp, we construct a family of 
connected graphs realizing these bounds in the next section.

\subsection{Construction of sparse irreducible symplectic positive definite matrices} \label{ss:TPn}

We can use basic symplectic matrices and  \Cref{sympd} to construct 
many sympPD matrices, including sparse  irreducible sympPD matrices.

\begin{prop}\label{p:IBOBsympPD}
Let $B=[b_{ij}]$ be a symmetric $p\times p$ matrix,  $S=\mtx{I & B\\O & I}$, and  $N=S\trans S$.
 Then  
$N=\mtx{ I & B \\
      B & I + B^2 }
   $
is a sympPD matrix. 

Now assume that $B$ is entrywise nonnegative. Then in 
$\G^L(N)$,  
  there is an edge joining $i$ and $j+p$ if and only if $b_{ij}\neq 0$ for $1\le i,j\le p$,
and an edge joining $i+p$ and $j+p$  if and only if there is a $k$ such that both $b_{ik}$ and $b_{jk}$ 
are nonzero for $1\le i<j\le p$.
\end{prop}

\begin{proof}
{We know $N=\mtx{ I & B \\
      B & I + B^2 }
   $ is a sympPD matrix because symplectic matrices are closed under products.}
 {The description of the edges of 
$\G^L(N)$} follows by noting that if $B$ is nonnegative, then the $(i,j)$-entry of 
$B^2$ is nonzero if and only if there is a $k$ such that both $b_{ik}$ and $b_{jk}$ are nonzero. 
\end{proof} 

An interesting example uses $B$ to be the $p\times p$ matrix  $B_p$ obtained from the adjacency matrix of the path on 
$p$ vertices by changing the $(1,1)$-entry to  a $1$.  The labeled graph of $N$ 
is connected, has {order}  $n=2p$  and $3p-2=(3/2)n-2$ edges.  
For example, when $p=5$,
\[ 
B_{5}= 
\mtx{
1 & 1 & 0 & 0 & 0\\ 
1& 0 & 1 & 0 &0 \\
        0 & 1 & 0 & 1 & 0 \\ 0 & 0 & 1 & 0 & 1\\
        0 & 0 & 0 & 1 & 0}\mbox{ and }
B_{5}^2=
\mtx{
2 & 1 & 1 & 0 & 0\\ 
1& 2 & 0 & 1 &0 \\
1 & 0 & 2 & 0 & 1 \\ 
     0 & 1 & 0 & 2 & 0\\
        0 & 0 & 1 & 0 & 1}.
\]
The labeled graph of $N$ for arbitrary $p$ is described in the  {next definition}. 
\begin{defn}\label{d:TPn}
The \emph{triangular path} $\tp_n$ on $n=2p\ge 4$ vertices is the  graph constructed as follows: Take $p-1$ copies of $K_3$, called $G_i,i=1,\dots,p-1$ and let $\hat G_1=G_1$.
For $i=1,\dots,p-2$, create $\hat G_{i+1}$ by identifying a degree-two vertex of $\hat G_i$ in  $G_i$   with a vertex of $G_{i+1}$. This results in $\hat G_{p-1}$, which has order $3(p-1)-(p-2)=2p-1$.  Then create $\tp_n$ from $\hat G_{p-1}$ by appending a leaf to a degree-two vertex of $\hat G_{p-1}$ in  $G_1$.  The \emph{standard labeling of $\tp_n$} is shown in   Figures~\ref{fig:tripathodd} and \ref{fig:tripatheven}  and is denoted by $\tp_n^I$.
\end{defn}

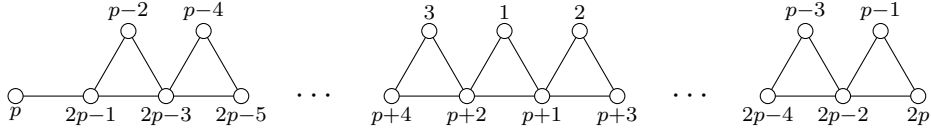
\begin{figure}[H]
\centering
\begin{tikzpicture}
\draw(0,0)--(3,0)--(2.5,{sqrt(3)/2})--(2,0)--(1.5,{sqrt(3)/2})--(1,0);
\node[below]at(0,0){$\substack{p}$};\node[below]at(1,0){$\substack{2p-1}$};\node[below]at(2,0){$\substack{2p-3}$};\node[below]at(3,0){$\substack{2p-5}$};
\node[above]at(1.5,{sqrt(3)/2}){$\substack{p-2}$};\node[above]at(2.5,{sqrt(3)/2}){$\substack{p-4}$};
\node at(4,0){$\dotsb$};
\draw(5,0)--(8,0)--(7.5,{sqrt(3)/2)})--(7,0)--(6.5,{sqrt(3)/2})--(6,0)--(5.5,{sqrt(3)/2})--(5,0);
\node[below]at(5,0){$\substack{p+4}$};\node[below]at(6,0){$\substack{p+2}$};\node[below]at(7,0){$\substack{p+1}$};\node[below]at(8,0){$\substack{p+3}$};
\node[above]at(5.5,{sqrt(3)/2}){$\substack{3}$};\node[above]at(6.5,{sqrt(3)/2}){$\substack{1}$};\node[above]at(7.5,{sqrt(3)/2}){$\substack{2}$};
\node at(9,0){$\dotsb$};
\draw(10,0)--(12,0)--(11.5,{sqrt(3)/2)})--(11,0)--(10.5,{sqrt(3)/2)})--(10,0);
\node[below]at(10,0){$\substack{2p-4}$};\node[below]at(11,0){$\substack{2p-2}$};\node[below]at(12,0){$\substack{2p}$};
\node[above]at(10.5,{sqrt(3)/2}){$\substack{p-3}$};\node[above]at(11.5,{sqrt(3)/2}){$\substack{p-1}$};
\foreach\i in{0,1,2,3,5,6,7,8,10,11,12}{\draw[fill=white](\i,0)circle(0.1);}
\foreach\i in{1.5,2.5,5.5,6.5,7.5,10.5,11.5}{\draw[fill=white](\i,{sqrt(3)/2})circle(0.1);}
\end{tikzpicture}
\caption{Standard labeled triangular path $\tp_{2p}^I$   when $p$ is odd.}
\label{fig:tripathodd}
\end{figure}

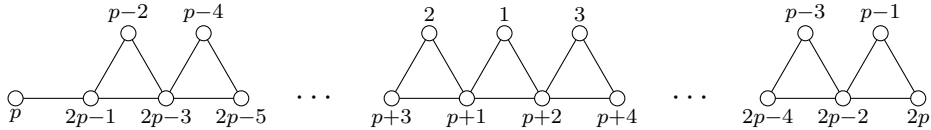
\begin{figure}[H]
\centering
\begin{tikzpicture}
\draw(0,0)--(3,0)--(2.5,{sqrt(3)/2})--(2,0)--(1.5,{sqrt(3)/2})--(1,0);
\node[below]at(0,0){$\substack{p}$};\node[below]at(1,0){$\substack{2p-1}$};\node[below]at(2,0){$\substack{2p-3}$};\node[below]at(3,0){$\substack{2p-5}$};
\node[above]at(1.5,{sqrt(3)/2}){$\substack{p-2}$};\node[above]at(2.5,{sqrt(3)/2}){$\substack{p-4}$};
\node at(4,0){$\dotsb$};
\draw(5,0)--(8,0)--(7.5,{sqrt(3)/2)})--(7,0)--(6.5,{sqrt(3)/2})--(6,0)--(5.5,{sqrt(3)/2})--(5,0);
\node[below]at(5,0){$\substack{p+3}$};\node[below]at(6,0){$\substack{p+1}$};\node[below]at(7,0){$\substack{p+2}$};\node[below]at(8,0){$\substack{p+4}$};
\node[above]at(5.5,{sqrt(3)/2}){$\substack{2}$};\node[above]at(6.5,{sqrt(3)/2}){$\substack1$};\node[above]at(7.5,{sqrt(3)/2}){$\substack{3}$};
\node at(9,0){$\dotsb$};
\draw(10,0)--(12,0)--(11.5,{sqrt(3)/2)})--(11,0)--(10.5,{sqrt(3)/2)})--(10,0);
\node[below]at(10,0){$\substack{2p-4}$};\node[below]at(11,0){$\substack{2p-2}$};\node[below]at(12,0){$\substack{2p}$};
\node[above]at(10.5,{sqrt(3)/2}){$\substack{p-3}$};\node[above]at(11.5,{sqrt(3)/2}){$\substack{p-1}$};
\foreach\i in{0,1,2,3,5,6,7,8,10,11,12}{\draw[fill=white](\i,0)circle(0.1);}
\foreach\i in{1.5,2.5,5.5,6.5,7.5,10.5,11.5}{\draw[fill=white](\i,{sqrt(3)/2})circle(0.1);}
\end{tikzpicture}
\caption{Standard labeled triangular path $\tp_{2p}^I$  when $p$ is even.}
\label{fig:tripatheven}
\end{figure}

 The graph $\tp_n$ can be thought of as a path of $p-1$ triangles with a leaf added.  Observe that the next two figures show  the same graph with the same standard labeling, but as indicated there,  which of $p+1$ and $p+2$  is closer to the leaf depends on whether  $p$ is even or odd.     It is shown in Corollary  \ref{tp_is_sympd}  that   $\GL L (N)=\tp_n^I$ for $N=\mtx{I & B_p\\B_p& I+B_p^2}$.

\begin{cor}
\label{tp_is_sympd}
For each integer $p\geq 2$, there exists a sympPD matrix in $ \symp(\tp_n^I)$.
\end{cor}

\bpf
Let $S=\begin{bmatrix}
    I & B_p\\
    O &I 
\end{bmatrix}$.
Then $N=S\trans S$ is a sympPD matrix.
It 
suffices to show that $\GL L(N)=\tp_{2p}^I$.

Observe that $N=\mtx{I & B_p\\B_p & I+B_p^2 
}$.  
Since $I$ is the leading $p\times p$ 
principal submatrix of $N$,
the set $\{1,\dots,p\}$ is an independent
set of vertices in $\G^L(N)$.  Edges in $\G^L(N)$  between $\{1,\dots,p\}$ and $\{p+1,\dots,2p\}$ are described by the nonzero entries of $B_p$. That is, the neighborhoods of vertices $1, \ldots, p$ are:  $N(1)=\{p+1,p+2\}$, $N(p)=\{2p-1\}$ and $N(i)=\{p+i-1,p+i+1\}$ for $i=2,\dots,p-1$.
Edges between vertices in $\{ p+1, \ldots, 2p\}$ are determined by the
dot products between columns of $B_p$. 
Thus, the edges in $\G^L(N)$ joining
pairs of vertices in $\{p+1, \ldots, 2p\}$ are $\{p+i, p+i+2\}$ for $i=1,\ldots, p-2$  and $\{p+1,p+2\}$. Hence $ \G^L(N)=\tp_{2p}^I$,
\epf

 We use $\tp_{2p}^\II$ to denote the coupled graph that this labeling represents, i.e. $\G^\II(N)$.


\subsection{Forbidden structures}\label{ss:forbid}

In this section  we seek conditions  that guarantee no matrix
with coupled graph $G^\PP$ is
sympPD.  First, we observe another property of sympPD matrices.

\begin{prop}\label{cor:bigraphblock}
{Let
$
\renewcommand{\arraystretch}{1.1}
N=\left[ \begin{array}{cc}
A & B \\  
B\trans & C
\end{array}
\right]$ be a sympPD matrix, where each block is $p\times p$.}
Then 
$B$ is similar to a symmetric matrix;  hence each eigenvalue of  $B$ is real
and its algebraic and  geometric multiplicities are equal.   
\end{prop}

\bpf
By \Cref{t:DJ},  $A$ is positive definite, and there is a symmetric matrix $W$ with $B=AW$.  Since $A$ is positive definite, $A$ has a positive definite square root $\sqrt{A}$.  Note that 
$\sqrt{A}^{-1}B\sqrt{A}= \sqrt{A}W\sqrt{A}$, 
so $B$ is similar to the symmetric matrix $\sqrt{A}W\sqrt{A}$. The remainder of the proof follows from properties of similarity and real symmetric matrices. 
\epf

Given an $n\x n$ positive definite matrix $N=[n_{ij}]$ with $n=2p$, it is sometimes useful to talk about the \emph{coupling digraph} of $N$, denoted by $\Gamma^\PP (N)$, which has {vertex set} $ \{v_1,\dots,v_p\}$ and arc set $\{(v_i,v_j): n_{i,j+p}\neq 0\}$.

\begin{rem}\label{rem:NOallequal}
Proposition \ref{cor:bigraphblock} can be applied to exhibit examples of coupled graphs that do not allow all symplectic eigenvalues to be equal, or equivalently do not allow a sympPD matrix.
\begin{enumerate}
\item \label{NOallequal1}
If  a strongly connected component of the digraph $\Gamma^\PP(N)$ is a directed cycle  of length $\ell \geq 3$,
then each matrix with  zero-nonzero pattern $B$ has 
at least two non-real eigenvalues.  Hence, no $N\in\symp(G^\PP)$  has all symplectic eigenvalues equal.
\item \label{NOallequal2}
If  $B$ has two irreducible components 
each equal to $[0]$ and there is a directed walk in $\Gamma^\PP(N)$ from one to the other, then every matrix with  zero-nonzero pattern $B$ has 0 as an eigenvalue with different algebraic and geometric multiplicities. 
Hence, no $N\in\symp(G^\PP)$  has all symplectic eigenvalues equal.
 \end{enumerate}
\end{rem}

These examples suggest considering the maximum number of entries in a coupled graph that does not allow all eigenvalues to be equal.  The next example exhibits representative labeling of a coupled graph of order $n=2p$ 
that omits $2p-1$ edges, i.e., it has $\binom{n}{2}-(n-1)=(p-1)(2p-1)$ edges.

\begin{ex}
\rm 
   {Let $n=2p$ and let $G$ be the  graph with  vertex set $\{v_1,\dots,v_{2p}\}$ and  edge set  equal to 
    the complement of $ \{\{v_i,v_{p+1}\}\colon i=1,\dots,p\}\,\cup\,
    {\{\{v_i,v_{p+2}\}\colon i=2,\dots,p\}}$ and let $\PP=\{(v_i,v_{i+p}):i=1,\dots,p\}$.    Let $N\in\symp(G^\PP)$. So $\Gamma^\PP(N)$  
    has two irreducible components each equal to $[0]$  and  an arc from one to the other. }
    Thus by Remark \ref{rem:NOallequal}\eqref{NOallequal2}, $N$ does not allow all symplectic eigenvalues to be equal.
\end{ex}

\begin{prop}\label{iso-isol}
    Let $G^\PP$ be a coupled graph that has an isolated vertex $u$ and let $N\in\symp(G^\PP)$.  If the vertex $\pp(u)$ is not isolated, then $N$ has at least two distinct symplectic eigenvalues. 
\end{prop}
\bpf Choose a representative labeling $L$ of $G^\PP$ in which $u$ is labeled 1. We assume that $N=[n_{ij}]\in\symp(G^L)$ is a sympPD matrix and prove that $\pp(u)$ (which is labeled $p+1$) is isolated.
  Recall that $(\Omega N)^2=-I$. The $(p+1)$-st row of $\Omega N$ is $(-n_{11},0,\dots,0)$, so the $(p+1)$-st  row of $(\Omega N)^2$ is $-n_{11}$ times the first row of $\Omega N$. Thus $((\Omega N)^2)_{p+1,j}=-n_{1,1} n_{p+1,j}$. Since $((\Omega N)^2)_{p+1,j}=0$ for $j\ne p+1$, this implies $n_{p+1,j}=0$ for $j\ne p+1$, i.e., the vertex labeled $p+1$ is isolated. 
\epf

 There are also sparse coupled graphs that do not allow a sympPD matrix, e.g., a single edge and isolated vertices with a coupling where an isolated vertex is  coupled to an endpoint of the $K_2$.  Any tree is connected and if its order is at least four, then it has too few edges  {to allow a sympPD matrix} by \Cref{c:sympPD-sparseLB}.

\section{The Strong Symplectic Spectral Property (SSSP) and applications}\label{s:strong} 

 Like the IEP-$G$, 
 the ISEP-$G$ may not behave well with respect to insertion of edges.  Namely, if  $A$ 
is a positive definite matrix and $H^\II$ is a coupled graph obtained from  $ \G^\II(A)$ by inserting some edges, then there may 
not be a matrix $B$ with coupled graph $H^\II$ having the same  symplectic spectrum as $A$.  
 In this section, using the  Strong Spectral Property for the IEP-$G$ as a model,  we develop an additional constraint for the symplectic setting and establish  the desired behavior in  \Cref{supergraph} (the Supergraph Theorem), establish other extensions of IEP-$G$ results that use the Strong Spectral Property,  and apply these ideas to give several results about the  ISEP-$G$, including that every coupled graph $G^\PP$ of order $2p$ allows every set of $p$ distinct positive numbers as a symplectic spectrum.

\subsection{SSSP foundations}

{We begin with a definition of the Strong Symplectic Spectral Property, stated in its most useful form (as is done with other strong properties). We then briefly discuss  the underlying algebra that provides insight, deferring the proof of the equivalence of the associated  form of the definition  to \Cref{appendix}.}

\begin{defn}\label{d:SSSP}
The $2p\times 2p$ positive definite matrix $N$  has the {\it Strong Symplectic Spectral Property} (SSSP) provided
$Y=O$   is the only symmetric $2p\times 2p$ 
matrix for which $N\circ Y=O$ and {$\Omega NY=YN \Omega$}.
\end{defn}

 The  $2p \times 2p$ matrices of the form 
\beq\label{Eq:HamForm}
M= \begin{bmatrix}
\begin{array}{cc}
R & E \\ 
F & -R^\top
\end{array} 
\end{bmatrix},
\eeq
where $R$ is a $p\times p$ real matrix, and  $E$ and $F$ are $p\times p$ real symmetric matrices
are called \textit{Hamiltonian matrices}. The set of $2p \times 2p$ Hamiltonian 
matrices {forms} a vector space, and is closed under the commutator $[\cdot, \cdot]$ (i.e., if $M$ and $M'$ are 
Hamiltonian matrices, then $[M,M']=MM'-M'M$ is a Hamiltonian matrix). 
Thus, the set of $2p\times 2p$ Hamiltonian matrices forms a Lie Algebra  with the commutator  as the bracket operator, and is denoted by $\spLA(2p)$. 
Additionally, it is known that if $M$ is a Hamiltonian matrix, then the exponential $e^M$ of $M$ is a symplectic matrix. Thus, the corresponding Lie group is $\Sp(2 p$), the group of $2p\times 2p$ symplectic matrices. 
In \Cref{SSSPequiv} it is shown that  $N$ having the SSSP 
is equivalent to 
\begin{equation}
\label{SSSPsum}
\{ NM+ M\trans N\colon M \in \spLA(2p) \} + \Span \mathcal{S}(\mathcal{G}^L(N))   = \symm(2p), 
\end{equation}}
 where $\symm(2p)$ denotes the set of $2p \x 2p$ symmetric matrices.

 It is not difficult to show that if $P_{\sigma}$ is a symplectic permutation matrix, then 
the positive definite matrix $N$ has the SSSP if and only if $P_\sigma N P_{\sigma}\trans$ 
has the SSSP. Thus,  any result about the SSSP  for graphs $\sym (G^L)$
holds for all labelings of $G$ that have the same coupling $\PP$.  It is often most convenient to choose a labeling of the coupled graph 
that captures the structure of the matrices in $\symp(G^\PP)$; 
and we do so in this section.

\begin{ex}
 Let $N=I_4$.  Then $Y=\begin{bmatrix}\begin{array}{rrrr}
0 & 0 & 0 & -1 \\
0 & 0 & 1 & 0 \\
0 & 1 & 0 & 0 \\
-1 &0 & 0 & 0 
\end{array} \end{bmatrix}
$
satisfies $\Omega I Y= YI\Omega$
and $N\circ Y=O$. Thus, $I_4$ 
does not have the SSSP. 
\end{ex}

If $N\in\symp(K_{2p}^{\PP})$ then $N$ has the SSSP (for any  coupling $\PP$ of $K_{2p}$), because $N\circ Y=O$ implies $Y=O$. 
There are other  coupled graphs for which every matrix has the SSSP, as seen in the next example.

\begin{ex}\label{ex:pK1veeKp-SSSP}
   Let  
$N\in\symp((\overline{K_{p}} \vee K_{p})^\MM)$ (cf. \Cref{pK1veeKp}).  Then $N=\mtx{E & B\\B\trans & C}$ where $E$ is diagonal and every entry of $B$ and $C$ is nonzero, so $N\circ Y=O$ implies $Y=\mtx{X& O\\O& O}$ where $X$ is a symmetric $p\x p$ matrix.  Then $\Omega NY=YN\Omega$ implies $-EX=O$ and thus $X=O$. 
\end{ex}

There are also coupled graphs for which no sympPD matrix has the SSSP, as seen in the next example. 
 \begin{ex}
    Let $G$ and $H$ be disjoint graphs of order $p\ge 2$ and let $\PP$ be a coupling such that $\pp(v)\in V(H)$ for every $v\in V(G)$. No sympPD matrix $N\in\symp((G\du H)^\PP)$  has the SSSP:  {Since $N$ is a sympPD matrix and $G$ and $H$ are disjoint, $N$ has the form $\mtx{B & O\\O & B^{-1}}$ by \Cref{t:DJ}.}
    Let $Y=\mtx{O & X\\X^T & O}$. Then $N\circ Y=O$  and $\Omega NY=Y N \Omega$ is  equivalent to $X\trans B=-BX$,  i.e.,  $X\trans B+BX=O$, which is a homogeneous system of    $\frac{p(p+1)}2$ linear equations in  $p^2$ variables. Thus there exists a nonzero solution $X$   and $N$ does not have the SSSP.
\end{ex}

In the next result, we    characterize the positive definite diagonal matrices that have the SSSP. For a positive integer  $p$, we denote by $E_{ij}$ the $2p\times 2p$
matrix with a  one in position $(i,j)$ and zeros elsewhere.

\begin{lem}\label{Diag}
Let 
$N= D \oplus E$ be a  $2p \times 2p$ positive definite diagonal matrix 
where $D=\diag(d_1, \ldots, d_p)$ and $E= \diag(e_1, \ldots, e_p)$. Then  the symplectic eigenvalues of 
$N$ are $\sqrt{d_ie_i}$, $i=1,\dots,p$, and $N$ has the SSSP if and only if the 
diagonal entries of $DE$ are distinct. 
In particular, if $D$ has distinct diagonal entries, then $\spspec(D\oplus D)=\{d_1, \ldots, d_p\}$ and $D \oplus D$ has SSSP. 
\end{lem}

\bpf
{By \Cref{ex:directsum}, the symplectic eigenvalues of  $N$ are $\sqrt{d_1e_1}, \ldots, \sqrt{d_pe_p}$.}
 Let 
$
Y= \left[ \begin{array}{c c} W & X \\  
X\trans & Z  \end{array} \right]
$
be a $2p\times 2p $ symmetric matrix such that 
$\Omega NY= YN \Omega$ and $N\circ Y=O$.  The latter implies that the diagonals  of $W$ and $Z$ are  zero,
while the former implies that
\begin{align}
EX\trans & = -XE,  \label{oneL} \\
-DX & = X\trans D, \label{twoL}\\
EZ& = WD, \mbox{ and } \label{threeL} \\
DW& = ZE. 
\label{fourL} 
\end{align}

First, observe that \eqref{oneL} requires that every diagonal entry of $X$ is zero. Pre-multiplying both sides of \eqref{oneL} by $D$ and substituting in 
\eqref{twoL} yields $DEX\trans= X\trans DE$. If the diagonal entries of $DE$ are distinct, then we conclude that the off-diagonal entries of $X$  are zero, i.e., $X=O$.

Let $W=[w_{ij}]$ and $Z=[z_{ij}]$. Examining the $(k,\ell)$ and $(\ell, k)$ entries of \eqref{fourL} and using the symmetry of $W$ and $Z$ gives
$d_kw_{k\ell}= e_\ell z_{k\ell}$, and 
$d_\ell w_{k\ell}= e_k z_{k\ell}$.
These equations imply $w_{k\ell}=z_{k\ell}=0$ or $d_k e_k=d_\ell e_\ell$. Thus, $N$ has the SSSP when the entries of $DE$ are distinct.

 Suppose that $N$ has a multiple symplectic eigenvalue. 
Without loss of generality, let $d_1e_1=d_2e_2$.
Let $Y= e_{2}(E_{12} + E_{21}) + d_{1}(E_{ p+2,p+1} +E_{ p+1,p+2})$. 
It can be verified that $Y$ is symmetric, $N\circ Y=O$, and $\Omega NY= YN\Omega $. 
\epf

\subsection{Fundamental theorems and consequences of the SSSP}\label{ss:SSSPthms}
In this section, we develop fundamental theorems about the SSSP.
The results and proofs are analogous to those for the Strong Spectral Property; 
we place  the more technical proofs in the appendix. 

The criteria given in \Cref{d:SSSP} for a 
positive definite matrix $N$
to have the SSSP  can be phrased 
in terms of a system of linear equations by viewing the
entries $y_{ij}$ with $i\leq j$ as distinct variables.
This leads to an effective purely linear algebraic way to determine whether $N$ has the SSSP.  
More precisely, 
one can verify that 
{\medmuskip=2mu \thickmuskip=3mu
\[ \{ E_{i, j+p}+ E_{j,i+p}: 1\leq i\leq j \leq p\}\cup
\{  E_{i+p,j}+E_{j+p,i}: 1\leq i\leq j \leq p\}
\cup \{  E_{ij}-E_{j+p,i+p} : 1\leq i,j \leq p\} 
\]}
is a basis for $\spLA(2p)$. Thus $\spLA (2p)$ has dimension $ 2\frac{p(p+1)}2 +p^2= 2p^2 + p $.
As we need an ordered basis, we order this first as the order of the three sets just defined.  Within each of the first two sets we order by the pair $(i,j)$ running through the main diagonal (in order), followed by each superdiagonal in order.  For the third, each superdiagonal is followed by the corresponding subdiagonal. This index order for a $4\x 4$ matrix (with $p=2$) is $(1,1), (2,2), (1 , 2)$ for each of the first two  sets and $(1,1), (2,2), (1,2), (2,1)$ for the third.  This ordered basis is called the \emph{standard ordered basis} (for verification). 

Given a $2p\times 2p$ symmetric matrix $M$ we denote by 
$\uvec(M)$  the $(2p^2+p) \times 1$ vector obtained from $M$ by stacking (in  order) 
the upper triangular parts of columns of $M$.
If $M=[m_{ij}]$ is $4 \times 4$, then
\[ \uvec(M)^\top=
\left[ \begin{array}{cccccccccc}
m_{11} & m_{12}& m_{22}& m_{13} & m_{23} & m_{33} &
m_{14}& m_{24} & m_{34} & m_{44} 
\end{array} \right].
\] 
Note that $\uvec$ defines a bijection 
between  $2p\x 2p$ symmetric matrices and $\mathbb{R}^{2p^2+p}$.

The {\it full SSSP verification matrix} of the $2p\times 2p$ positive definite matrix  $N$ is the  $ (2p^2+ p) \times (2p^2+p)$ matrix, $\Phi(N)$, whose columns are 
the vectors  $\uvec (M^\top N + N M)$ where $M$  runs over the standard ordered basis vectors of $\spLA(2p)$ {in order}. The \emph{SSSP verification matrix} of $N$, $\Xi(N)$,  is the submatrix of $\Phi(N)$ with rows  indexed by the pairs $(i,j)$  where $i<j$  and $ij$ is not an edge of $G$. Thus, $\Xi(N)$ has $2p^2+p-(2p+ |E(G)|)=
2p^2-p -|E(G)|$ rows.

 \begin{ex}\label{ex:verify-cycle}
 Consider a cycle on four vertices  with the vertices named in cycle order, so then $\OL{C_4}$ has  edges $u_1u_3$  and $u_2u_4$. Label this graph by $I(u_i)=i$ and 
denote the labeled graph by $C_4^I$. If $N\in\symp(C_4^I)$, then $n_{13}=n_{24}=0$ and all other entries are nonzero.  Since the columns of $\Phi(N)$ have the form $\uvec (M^\top N + N M)$ where $M$ runs through
$2E_{13}$, {$2E_{24}$}, $E_{14}+E_{23}$, $2E_{31}$, $2E_{42}$, $E_{32}+E_{41}$, $E_{11} - E_{33}$, $E_{22}-E_{44}$, $E_{12}- E_{43}$ and $E_{21}-E_{34}$, 
\beq\label{eq:P4SSSP}
\Xi(N)=
\left[\begin{array}{cccccccccc}
2 \, n_{11} & 0 & n_{12} & 2 \, n_{33} & 0 & n_{34} & 0 & 0 & -n_{14} & n_{23} \\
0 & 2 \, n_{22} & n_{12} & 0 & 2 \, n_{44} & n_{34} & 0 & 0 & n_{14} & -n_{23} 
\end{array}\right].
\eeq
Examination of the {first and second} columns of $\Xi(N)$ shows  that the rows of $\Xi(N)$ are independent.  Thus \Cref{t:verification} below shows that every positive definite matrix matrix $N$ with $\G^L(N)=C_4^I$ has the SSSP. { Observe that $C_4^I$ is a representative labeling of  $K_{2,2}^{\PP_1}$ with partite sets $\{v_1,v_2\}$, $\{v_3,v_4\}$ and coupling  $\PP_1=\{(v_1,v_2 ),(v_3,v_4)\}$.}
\end{ex}
\begin{thm} [{SSSP Verification Matrix}]
\label{t:verification}
Let $N$ be a $2p\times 2p$ positive definite matrix with labeled graph $G^L$. 
Then the following are equivalent
\ben[$(a)$]
\item \label{a:verification} $N$ has the SSSP.
\item \label{b:verification} 
The column space of $\Xi(N)$ spans $\mathbb{R}^{2p^2  -p } -|E(G)|$. 
\item \label{c:verification} The rows of $\Xi(N)$ are linearly independent. 
\een
\end{thm} 

 The proof is presented in \Cref{appendix}.
 It is more common to determine an SSSP verification matrix for a specific $N$, as in the next example.

\begin{ex}\label{KppSympPD-SSSP} 
Let 
$N=
\begin{bmatrix}
    1 & 0 & -1/\sqrt{2} & 1/\sqrt{2} \\0 & 1 & 1/\sqrt{2} & 1\sqrt{2} \\-1/\sqrt{2} & 1/\sqrt{2} & 2 & 0 \\1/\sqrt{2} & 1/\sqrt{2} & 0 & 2
\end{bmatrix}$
be  a $4\times 4$ symmetric matrix (observe that $N$ is sympPD by \Cref{KppSympPD}). 
 Then, using the same indexing as in the previous example, 
\[ \Xi(N)=\left[\begin{array}{rrrrrrrrrr}
0 & 0 & 0 & \sqrt 2 & \sqrt 2 & 0 & 0 & 0 & 1 & { 1} 
\\
\sqrt 2 & \sqrt 2 & 0 & 0 & 0 & 0 & 0 & 0 & -2 & -2 
\end{array}\right].
\] 
Here the rows of $\Xi(N)$ correspond to 
entries $(1,2)$ and $(3,4)$ of $N$.
Note that $\Xi(N)$ has rank 2, and hence $N$ has the SSSP.   Observe that $\G^\II(N)=K_{2,2}^\MM=K_{2,2}^{\PP_2}$.
\end{ex}

The next result shows that the set of positive definite $2p\times 2p$ matrices with the SSSP is an open subset of {$\symm(2p)$}. 
Here $\| M\|$ denotes any fixed 
induced matrix norm, e.g., the spectral norm.

\begin{cor}
\label{open}
Let $N$ be a $2p\times 2p$ positive definite matrix with the SSSP.
Then there exists $\epsilon>0$  such that every {symmetric} matrix   $M$ with 
 $\|N-M \|<\epsilon$ is positive definite and has the SSSP. 
\end{cor}

\begin{proof}
Choose $\epsilon_1$ sufficiently small so that $\|N-M\|<\epsilon_1$ ensures that $n_{ij}\ne 0$ implies $m_{ij}\ne 0$ for $i,j=1,\dots,2p$. Thus $\|N-M\|<\epsilon_1$ implies $\G(M)$ is supergraph of $\G(N)$. The eigenvalues of $N$ are  continuous functions of the entries of $N$, so there exists $\epsilon_2$ such that $\|N-M\|<\epsilon_2$ implies $M$ is positive definite.

Let $\alpha$ be defined by $\Xi(N)=\Phi(N)[\alpha]$. 
Choose  $\epsilon_3$ such that  $|(\Xi(N))_{ij}-(\Phi(M)[\alpha])_{ij}|<\epsilon_3$ for all $i\in \alpha$ and $j=1,\dots,2p^2+p$  implies the rows of $\Phi(M)[\alpha]$ are linearly independent.  Since the entries of $\Xi(N)$ are   continuous functions of the entries of $N$, $\epsilon\le \min\{\epsilon_1,\epsilon_2,\epsilon_3\}$ can be chosen to guarantee that $\|N-M\|<\epsilon$ implies $|(\Xi(N))_{ij}-(\Phi(M)[\alpha])_{ij}|<\epsilon_3$.
 Since $\Xi(M)$ is obtained from $\Phi(M)[\alpha]$ by deleting zero or more rows,  $\|N-M\|<\epsilon$  implies the rows of $\Xi(M)$ are linearly independent, so   
 $M$ has the SSSP by  \Cref{t:verification}. 
\end{proof}

The next  result shows that if a coupled graph allows a positive definite matrix 
with the SSSP and a given symplectic spectrum, then each coupled graph obtained by inserting edges allows a 
positive definite matrix with the same symplectic spectrum.

\begin{thm}{\rm \bf (Supergraph Theorem)}
\label{supergraph}
\\
Let $N$ be a  $2p\times 2p$ positive definite matrix having the SSSP and coupled graph $G^\II:=\G^\II(N)$. 
Let $H^\II$ be a coupled graph on $2p$ vertices obtained from $G^\II$ by inserting some edges.
Then there exists a matrix  with  
coupled graph $H^\II$ having the same symplectic eigenvalues as $N$ and also having the SSSP.
\end{thm}

The proof of the Supergraph Theorem is similar to that used for the analogous IEP-$G$ result and can be found in \Cref{appendix}.

\begin{cor}\label{AllSimple}
    Every coupled graph  $G^\PP$ of order $n=2p$ allows 
    every set of $p$ distinct positive real numbers as the symplectic spectrum of {some  $N\in \symp(G^\PP)$} with the SSSP. 
\end{cor}
\begin{proof}
Let $\lam_1, \ldots, \lam_p$ be distinct positive real numbers, and $D= \diag(\lam_1, \ldots, \lam_p)$.
Then by \Cref{Diag}, $D \oplus D$ is a positive definite matrix with the SSSP whose coupled graph is 
the empty graph on $2p$ vertices.  The result now follows from \Cref{supergraph}.
\end{proof}

The next theorem  is an analog of a main result for the Strong Spectral Property and the IEP-$G$.  Its proof is in \Cref{appendix}.

\begin{thm}\label{Thm:bifurcation}{\rm \bf (Bifurcation Theorem) }
\\
Let $N$ be a $2p \times 2p$ positive definite matrix having the SSSP.  
Then there exists $\epsilon>0$ such that if $\widehat{N} $ is a $2p\times 2p $ positive definite matrix 
with $\| \widehat{N}- N \| <\epsilon$, then there exists a 
matrix  $N'\in \symp(\G^\II(N))$ such that $\spspec(N')=\spspec(\wh N)$.
\end{thm}

\begin{cor}\label{c:bifur-close}
Let $N$ be a $2p\times 2p $ positive definite matrix  \ with the SSSP and $\spspec(N)=\{\lam_1,\dots,\lam_p\}$. 
Then there exists $\delta>0$   {with the property that}   for every list $\alpha=\{\hat\lam_1,\dots,\hat\lam_p\}$ of $p$ positive real numbers such that $|\hat\lam_i-\lam_i|<\delta$ for $i=1,\dots,p$,  there is a matrix $N'\in \symp(\G^\II(N))$ such that $\spspec(N')=\alpha$.
\end{cor}
\bpf We use the spectral norm.
Let $S$ be a symplectic matrix such that $S\trans NS=D\oplus D$ where $D=\diag(\lam_1,\dots,\lam_p)$.
Use \Cref{Thm:bifurcation} to choose $\epsilon>0$  such that if $\widehat{N} $ is a $2p\times 2p $ positive definite matrix 
with $\| \widehat{N}- N \| <\epsilon$, then there exists a matrix  $N'\in\symp(\G^L (N))$ such that $\spspec(N')=\spspec(\widehat{N})$. Define $\delta=\frac\epsilon{\|S^{-1}\|^2}$. Let $\alpha=\{\hat\lam_1,\dots,\hat\lam_p\}$ be such that $|\hat\lam_i-\lam_i|<\delta$ for $i=1,\dots,p$.  Define $\wh D=\diag(\hat \lam_1,\dots,\hat \lam_p)$ and $\wh N=(S^{-1})\trans(\wh D\oplus \wh D)S^{-1}$. Observe that $\spspec(\wh N)=\alpha$.  Then 
\begin{align}\nonumber
\|\wh N-N\|&=\|((S^{-1})\trans((\wh D-D)\oplus (\wh D-D)) S^{-1}\|\\ 
\nonumber
&\le \|(S^{-1})\trans \| \|(\wh D-D)\oplus (\wh D-D)\| \|S^{-1}\|< \| S^{-1}\|^2\delta=\epsilon.
\end{align} 
Thus by \Cref{Thm:bifurcation}, there exists a positive definite  matrix $N'$  with $\G^\II(N')=\G^\II(N)$ and $\spspec(N')=\spspec(\widehat{N})=\alpha$. \epf

  Given a positive definite matrix $N$ we list its distinct symplectic eigenvalues in increasing order 
$\lambda_1<  \lambda_2< \cdots< \lambda_q$.
  The \textit{ordered symplectic multiplicity list}  of $N$ 
is  $\soml(N)=(m_1, \ldots, m_q)$  where $m_i$ is the multiplicity of $\lambda_i$
as a symplectic eigenvalue of $N$.  A \emph{refinement} of $\soml(N)$ is a list obtained from $\soml$ by replacing
each $m_i$ with a partition of $m_i$. 
 The next result is immediate from \Cref{c:bifur-close}.

\begin{cor}\label{cor:bifurcationSMP}{\bf (Bifurcation Corollary for Ordered Symplectic Multiplicity Lists)}
\noi
    Let $N$ be a positive definite matrix with the SSSP.  Then for each refinement $\soml'$ of $ \soml(N)$, there is a positive definite matrix $N'$ with the SSSP, { $\G^\II(N')=\G^\II(N)$,}  and $\soml({N'}) = \oml'$. 
In particular,    if all the symplectic eigenvalues of $N$ are equal, then {$\G^\II(N)$}  allows every possible ordered symplectic multiplicity list. 
\end{cor}


\section{Symplectic spectrally arbitrary labeled graphs}\label{s:ssa}

A coupled graph $G^\PP$ of order $n=2p$  is \emph{symplectic spectrally arbitrary} if, for any $p$ positive numbers $\lam_1,\dots,\lam_p$, there is a   matrix $N\in\symp(G^\PP)$ such that $\lam_1,\dots,\lam_p$ are the symplectic eigenvalues of $N$. 
{Observe that
     any coupled graph that does not allow a sympPD matrix  is not symplectic spectrally arbitrary. In particular, if $G$ is a connected graph of order $n=2p$ and $|E(G)|<3p-2$, then $G^\PP$ is not symplectic spectrally arbitrary for any coupling $\PP$ of $G$ by \Cref{c:sympPD-sparseLB}. }

In this section, we use the construction $N=S\trans D'S$ where  $S$ is a  symplectic matrix  and  $D'=D\oplus D$ is a diagonal matrix
with a desired symplectic spectrum  to construct  a variety of coupled graphs that are symplectic spectrally arbitrary. In particular, we  show that the sparse coupled graph  $\tp_n^\II$ defined in \Cref{ss:TPn} is symplectic spectrally arbitrary 
for each positive even integer $n=2p\ge 4$. 
Additionally, we show that a  variety of dense coupled graphs, including  $K_{n}$ with any coupling, 
are symplectic spectrally arbitrary.

\begin{prop}\label{l:nonnegsymplectic}
 Let $S$ be a $2p\times2p$ nonnegative symplectic  matrix. 
  Then $\G^\II(S^\top S)$ is symplectic spectrally arbitrary.
\end{prop}

\begin{proof} Let $D$ be a $p\x p$ diagonal matrix with positive diagonal entries,  let $D'=D\oplus D$, and define $N=S^\top D'S$. Then  the  symplectic eigenvalues of $N=[n_{ij}]$ are the diagonal entries of $D$. Note that \[ 
n_{ij}=\sum_{r=1}^{2p} s_{r i}d'_{rr}s_{r j}. 
\]
Since all entries of $S$ are nonnegative and all entries $d'_{rr}$ are positive, $n_{ij}\neq0$ if and only if there exists $1\leq r\leq 2p$ such that $s_{r i}s_{r j}\neq0$. This is equivalent to $(S^\top S)_{ij}=\sum_{r=1}^n s_{r i}s_{r j}\neq0$. Hence, $\G^\II(N)=\G^\II(S\trans S)$ and $\G^\II(S\trans S)$ is symplectic spectrally arbitrary.
\end{proof}

\begin{rem}\label{manySSA}
\Cref{l:nonnegsymplectic} together with the indicated prior result shows that   each of the following  coupled graphs is symplectic spectrally arbitrary for every $p\ge 2$: 
$(\OL{K_p}\vee K_p)^{\MM}$ (\Cref{pK1veeKp}) and 
    $(\tp_{2p})^\II$ (\Cref{tp_is_sympd}).
\end{rem}

{We can strengthen the result in \Cref{p:IBOBsympPD}. 

\begin{prop}\label{IBOIisSSA} Let $B=[b_{ij}]$ be a symmetric $p\times p$ matrix,  $S=\mtx{I &  B\\O & I}$, and  $N=S\trans S$.  Then $\G^\II(N)$ is symplectic spectrally arbitrary.
\end{prop}
\bpf 
For {positive real numbers} $\lam_1,\dots,\lam_p$,  define $D=\diag(\lam_1,\dots,\lam_p)$ and $D'=D\oplus D$.  Use $B$ and $D$ to define 
    $\wt B=\sqrt D^{-1}B \sqrt D^{-1}$,   $\wt S=\mtx{I & \wt B\\O & I}$, and $\wt N=\wt S\trans D' \wt S$.  Then $\wt S$ is symplectic and  $\spspec({\wt N})=\{\lam_1,\dots,\lam_p\}$.  Furthermore,
\[ 
\wt N=\wt S\trans D' \wt S=\mtx{D & D\wt B \\
       \wt B D &  D+\wt B D\wt B }=
       \mtx{D & \sqrt D B \sqrt D^{-1}\\
       \sqrt D^{-1} B\sqrt D & D+\sqrt D^{-1}B^2 \sqrt D^{-1}}. 
\] 
 Since $N=\mtx{I & B \\B & I+B^2}$ and $D$ is diagonal, $\G^\II(\wt N)=\G^\II(N)$.
\epf
The next result is immediate from \Cref{IBOIisSSA} and \Cref{KppSympPD}.
\begin{cor}\label{KppSSA}
The coupled graph $K_{p,p}^\MM$ is  symplectic spectrally arbitrary.
 \end{cor}
}

In the next two examples, we create positive definite matrices with specified sets of symplectic eigenvalues by using `random' symplectic matrices to `smear' the entries of an appropriate diagonal matrix.

\begin{ex}\label{random2} 
Define the coupled graph $(K_{p}\du K_p)^\II$ by the requirement that $v$ and $\pp(v)$ are in different connected components of $K_{p}\du K_p$     for every vertex $v$. 
Given an arbitrary multi-set $\{\lam_1,\dots,\lam_p\}$,  let $D=\diag(\lam_1,\dots,\lam_p)$, let $D'=\mtx{D  & O\\ O&D}$, and let $A$ be a random  $p\x p$ matrix.  With high probability, $A$ is invertible.  Then   $S_A=\mtx{A & O\\O & (A^\top )^{-1}}$ is a symplectic matrix and $S_A^\top D'S_A=\mtx{A^\top D A & O\\ O & (A^{-1})D (A^\top )^{-1}}$ is 
 a positive definite matrix with symplectic eigenvalues $\lambda_1, \ldots, \lambda_p$.
Since $A$ is chosen randomly for this set of symplectic eigenvalues,     $G^\II(S_A^\top D'S_A)\cong (K_p\du K_p)^\II$ with high probability.  Thus the coupled graph  $(K_p\du K_p)^\II$  is symplectic spectrally arbitrary. 
\end{ex}

\begin{ex}\label{random} Let $\PP$ be any coupling of $K_{2p}$.
Given an arbitrary multi-set $\{\lam_1,\dots,\lam_p\}$,  let $D=\diag(\lam_1,\dots,\lam_p)$, let $D'=\mtx{D  & O\\ O&D}$,  let $A$ be a random  $p\x p$ matrix, and let $B$ be a random $p\x p$ symmetric matrix, i.e., the entries of $B$ on the diagonal and above are chosen randomly. Then  $R_B=\mtx{I & B\\O & I}$ and $S_A=\mtx{A & O\\O & (A^\top )^{-1}}$ are symplectic matrices,  and $R_B^\top S_A^\top D'S_AR_B=\mtx{A^\top D A & A^\top D AB\\B A\trans D A & BA^\top D AB+(A^{-1})D (A^\top )^{-1}}$. 
Since $A$ and $B$ are chosen randomly, {$\G^\II((S_AR_B)^\top D'S_AR_B)=K_{2p}^\PP$}
with high probability.  
 \end{ex}

As noted in \Cref{s:strong},  $N\in\symp(K_{2p}^\PP)$  implies $N$ has the SSSP for any coupling $\PP$ of $K_{2p}$.  



\section{Symplectic  eigenvalues of order four graphs}\label{s:order4}

 In this section we determine the possible symplectic eigenvalues  of each coupling of the eleven graphs of order four.  To do this, we develop several additional tools for determining symplectic spectra, some of which are applications of strong properties. We begin with a remark that greatly simplifies this problem by applying the machinery we have developed.

\begin{rem}\label{o4sympPD3SSA}
    For $n=4$ there are only two symplectic eigenvalues and thus only two possible multiplicity cases: two equal or two distinct symplectic eigenvalues.  Since every coupled graph allows all possible distinct symplectic spectra (by \Cref{AllSimple}), solving the ISEP-$G$ for a coupled graph $G^\PP$ of order four is equivalent to determining whether $G^\PP$ allows a sympPD matrix; if so, it is symplectic spectrally arbitrary. 

 Thus  every coupling of    $K_4$ is symplectic spectrally arbitrary {(with the SSSP)} by \Cref{random}. 
 
By \Cref{c:sympPD-sparseLB}, any connected graph on four vertices with fewer than four  edges does not allow a sympPD and thus has all symplectic eigenvalues simple.  Therefore, each coupling of $P_4$ and $K_{1,3}$ allows any set of two distinct symplectic eigenvalues and no other symplectic spectra.
Recall that the three distinct couplings for a graph  with vertex names $v_1,v_2,v_3,v_4$ are $\PP_1=\{(v_1,v_2),(v_3,v_4)\}$, $\PP_2=\{(v_1,v_3),(v_2,v_4)\}$, $\PP_3=\{(v_1,v_4),(v_2,v_3)\}$.
\end{rem}

Among connected graphs, this leaves  the paw 
graph, shown in  \Cref{fig:paw-couple}, for which the symplectic eigenvalues need to be determined for various couplings,  and  certain couplings of   $K_4-e\cong {\OL{K_2}}\vee K_2$  and $C_4\cong K_{2,2}$.

\begin{ex}\label{ex:paw}
Consider the paw graph, which has  $n= 4$ and  $p=2$.
 One representative labeling for each coupling $\PP_i,i=1,2,3$ as defined in \Cref{r:ord4couplings} is shown in  \Cref{fig:paw-couple}; we have chosen a representative with the degree three vertex always labeled 1. Observe that the labelings for $\PP_2$ and $\PP_3$ determine the same pattern of nonzero entries in the matrices, so there are only two cases, with the pendent vertex labeled 2 or 3. 

\begin{figure}[H]
\begin{center}
\begin{tikzpicture}   
\draw(1,0)--(0,0.5)--(0,-0.5)--(1,0)--(2,0); \draw[fill=white](1,0)circle(0.1); \node[below]at(1,-0.1){$v_1$}; \draw[fill=white](2,0)circle(0.1); \node[below]at(2,-0.1){$v_2$}; \draw[fill=white](0,0.5)circle(0.1); \node[above right]at(0,0.5){$v_3$}; \draw[fill=white](0,-0.5)circle(0.1); \node[below right]at(0,-0.5){$v_4$};
\node[above]at(1,0.1){$1$}; \node[above]at(2,0.1){$3$}; \node[left]at(-0.1,0.5){$2$}; \node[left]at(-0.1,-0.5){$4$}; \node at(1,-1.5){$L_1, \PP_1$};

\begin{scope}[shift={(4,0)}]
\draw(1,0)--(0,0.5)--(0,-0.5)--(1,0)--(2,0); \draw[fill=white](1,0)circle(0.1); \node[below]at(1,-0.1){$v_1$}; \draw[fill=white](2,0)circle(0.1); \node[below]at(2,-0.1){$v_2$}; \draw[fill=white](0,0.5)circle(0.1); \node[above right]at(0,0.5){$v_3$}; \draw[fill=white](0,-0.5)circle(0.1); \node[below right]at(0,-0.5){$v_4$};
\node[above]at(1,0.1){$1$}; \node[above]at(2,0.1){$2$}; \node[left]at(-0.1,0.5){$3$}; \node[left]at(-0.1,-0.5){$4$}; \node at(1,-1.5){$L_2, \PP_2$};
\end{scope}

\begin{scope}[shift={(8,0)}]
\draw(1,0)--(0,0.5)--(0,-0.5)--(1,0)--(2,0); \draw[fill=white](1,0)circle(0.1); \node[below]at(1,-0.1){$v_1$}; \draw[fill=white](2,0)circle(0.1); \node[below]at(2,-0.1){$v_2$}; \draw[fill=white](0,0.5)circle(0.1); \node[above right]at(0,0.5){$v_3$}; \draw[fill=white](0,-0.5)circle(0.1); \node[below right]at(0,-0.5){$v_4$};
\node[above]at(1,0.1){$1$}; \node[above]at(2,0.1){$2$}; \node[left]at(-0.1,0.5){$4$}; \node[left]at(-0.1,-0.5){$3$}; \node at(1,-1.5){$L_3, \PP_3$};
\end{scope}
\end{tikzpicture}\\
\caption{For $i=1,2,3$, a representative labeling $L_i$ {of Paw for}  the coupling $\PP_i$ {is shown.}
\label{fig:paw-couple}}
\end{center}
\end{figure}

\noi\underline{Coupling $\PP_1=
\{(v_1,v_2),(v_3,v_4)\}$ {and labeling $L_1$ shown in \Cref{fig:paw-couple}}:} The pendent vertex is labeled $3$.
Then positive definite matrices {$\symp({\rm Paw}^{L_1})$}  have the form 
$
\left[\begin{array}{rrrr}
a & f & e & g \\
f & c & 0 & h \\
e & 0 & b & 0 \\
g & h & 0 & d
\end{array}\right]$.
One can verify that 
$
(\Omega N)^2= 
\left[\begin{array}{rrrr}
-a b + e^{2} & -b f & 0 & -b g \\
-d f + e g + g h & -c d + h^{2} & b g & 0 \\
0 & e f + c g - f h & -a b + e^{2} & -d f + e g + g h \\
-e f - c g + f h & 0 & -b f & -c d + h^{2}
\end{array}\right]$. 
In particular, the $(1,4)$-entry is nonzero. 
Hence $(\Omega N)^2 +tI \neq O$ for any $t$.  Hence every positive definite matrix 
with this partitioned graph has distinct symplectic eigenvalues.
\bigskip

\noi\underline{Coupling $\PP_2=\{(v_1,v_3),(v_2,v_4)\}$ {and labeling $L_2$ shown in \Cref{fig:paw-couple}}:} 
The pendant vertex  is labeled 2. 
 The leading principal minors of
\[ 
 N=
\left[ \begin{array}{rrrr}
 3 &-1 &  1&  1 \\
-1 &  1 &  0&   0 \\
1 &  0&  1 & 1\\
 1 & 0 & 1 & 2
 \end{array} \right]
\] 
 equal 3, 2, 1, and 1, so 
$N$ is positive definite.  
Also, one can verify that  $(\Omega N)^2=-I$. Hence, {each symplectic eigenvalue of $N$ is one.} 
 {Since $\G^L(N)={\rm Paw}^{L_2}$,} this labeling of the  paw graph is symplectic spectrally arbitrary.

Additionally, $N$ has the SSSP.  To see this, 
note that if $Y$ is a matrix of the form
$ \left[\begin{array}{rrrr}
0 & 0 & 0 & 0 \\
0 & 0 & a & b \\
0 & a & 0 & 0 \\
0 & b & 0 & 0
\end{array}\right] 
$
such that 
$\Omega NY= YN\Omega$, then 
$ \left[\begin{array}{rrrr}
0 & a + b & 0 & 0 \\
a + b & 2 \, a + 4 \, b & -a - b & 0 \\
0 & -a - b & 2 \, a & -a + b \\
0 & 0 & -a + b & -2 \, b
\end{array}\right]=O$.
This occurs only if $Y=O$, and hence $N$ has the SSSP.  
\end{ex}

\begin{ex}\label{K22cup2}
 {As we have done before, n}ame the vertices of $K_{2,2}$ so that the partite sets are $\{v_1,v_2\}$ and  $\{v_3,v_4\}$. {Then   $K_{2,2}^\MM= K_{2,2}^{\PP_2}$} was shown to allow a sympPD matrix {with the SSSP} in \Cref{KppSympPD-SSSP}. By  symmetry, the coupling $\PP_3$ of $K_{2,2}$ has a representative representative labeling that allows a sympPD matrix {with the SSSP}.

For the coupling $\PP_1=\{(v_1,v_2),(v_3,v_4)\}$, consider the matrix
\begin{align}
\label{eq:PP2}
 N=
 \mtx{
 \begin{array}{rrrr}
2 & 1 & 0 & 1 \\
 1 & 2 & 1 & 0 \\
 0 & 1 & 2 & -1 \\
 1 & 0 & -1 & 2 
 \end{array}
 }.
\end{align} 
The (standard) eigenvalues of $N$ are $2+\sqrt 2, 2+\sqrt 2,2-\sqrt 2,2-\sqrt 2$, so  $N$ is positive definite.
Also, one can verify that $(\Omega N)^2=-2I$, so $\spspec(N)=\{\sqrt 2,\sqrt 2\}$.
  Thus this labeling of $K_{2,2}$ is symplectic spectrally arbitrary. Note that the labeling used here is that of $C_4^I$  in \Cref{ex:verify-cycle}, where it is shown that every  matrix of this form  has the SSSP,   and hence all labelings of $C_4\cong K_{2,2}$ are  symplectic spectrally arbitrary and for each labeling there is a sympPD matrix with the SSSP. 
\end{ex}

\begin{ex}\label{K4-ecup2}
Name the vertices of $\OL{K_2}\vee K_2$ so that  $v_1$ and $v_2$ are not adjacent.
Then   {$(\OL{K_2}\vee K_2)^\MM=(\OL{K_2}\vee K_2)^{\PP_2}$} was shown to allow a sympPD matrix {with the SSSP} in \Cref{ex:pK1veeKp-SSSP}. By  symmetry, the coupling $\PP_3$ of $K_{2,2}$ has the same property.

For the coupling $\PP_1=\{(v_1,v_2),(v_3,v_4)\}$, we use the labeling $L_1(v_1)=1$, $L_1(v_2)=3$, $L_1(v_3)=2$, $L_1(v_4)=4$ and observe that  $(\OL{K_2}\vee K_2)^{L_1}$ is a supergraph of $C_4^I$.  Since  a sympPD matrix $N\in\symp(C_4^I)$ with the SSSP was exhibited in \Cref{K22cup2}, \Cref{supergraph} implies that $(\OL{K_2}\vee K_2)^{L_1}$ allows a sympPD matrix and is therefore symplectic spectrally arbitrary. {Hence all labelings of $K_4-e\cong \OL{K_2}\vee K_2$ are  symplectic spectrally arbitrary and for each labeling there is a sympPD matrix with the SSSP.} \end{ex}

  Next we consider the five unlabeled disconnected graphs of order four: $4K_1$, $2K_1\du K_2$, $2K_2$, $K_1\du P_3$, and $K_1\du K_3$.
The graph $4K_1$ is  clearly symplectic spectrally arbitrary. The graphs $K_1\du P_3$ and $K_1\du K_3$ must have all simple symplectic eigenvalues by \Cref{iso-isol}.
For the remaining two disconnected graphs of order four, we consider how the coupling interacts with the connected components. Let $\GC G C$ be a coupled graph with  subgraph $H$.   The coupling \emph{respects $H$} if $\pp(v)\in {V(H)}$  for every $v\in V(H)$; in this case, the \emph{coupling $\PP$ restricted to $H$} is the coupling $\PP|_{H}$ of $H$ where $\pp|_{H}(v)=\pp(v)$.  We first develop a matrix tool.

 \begin{thm}
\label{directsum}
Let $P$ and $Q$ be  $2m \times 2m$ 
and $2r \times 2r$  positive definite matrices
with block forms
$
P=\left[ \begin{array}{cc}
P_1  & P_2 \\ 
P_2\trans & P_3 \end{array} \right]$,
$Q=\left[ \begin{array}{cc}
Q_1  & Q_2 \\ 
Q_2\trans & Q_3 \end{array} \right]$
where $P_1$ is $m\times m$ (respectively, $Q_1$ is 
$r \times r$), 
and let 
\[ N=\left[ 
\begin{array}{cccc}
P_1 & O & P_2 & O \\  
O & Q_1 & O & Q_2 \\ 
P_2\trans & O &P_3 & O \\  
O & Q_2\trans & O & Q_3
\end{array} \right]. 
\] 
Then: 
\ben[$(a)$]
\item 
\label{directsum-a}
$N$ is positive definite.
\item
\label{directsum-b}$
\spspec(N)=\spspec(P) \cup \spspec(Q)$.
\item
\label{directsum-c}   $N$ has the SSSP if and only if  $P$ and $Q$ have the SSSP and  $\spspec(P) \cap \spspec(Q)= \emptyset.$
\een
\end{thm}

\begin{proof} 
Since $N$ is permutationally similar to $P\oplus Q$, \eqref{directsum-a} holds. 

Let $p=m+r$, 
$R= \mtx{
I_m & O & O & O \\ 
O & O & I_r & O \\  
O & I_m & O & O \\ 
O & O & O & I_r
} \!\!, \ \widehat{\Omega}= R\trans \Omega R, \  \mbox{and} \ \widehat{N}= R\trans N R$. 
Then 
\beq \label{eq1} \widehat{\Omega}=
\Omega_m \oplus \Omega_r \ \mbox{ and }\  \widehat{N}= P \oplus Q .\eeq
Since $R$ is a permutation matrix,  
\[ \spec(\Omega N)=\spec(R\trans \Omega N R)=\spec(\wh \Omega \wh N) =\spec((\Omega_m P)\oplus ( \Omega_r Q))=\spec(\Omega_m P)\cup \spec(\Omega_r Q).\]
Therefore, \eqref{directsum-b} holds.

For a symmetric  $2p\x 2p$ matrix $Y$, define
$\widehat{Y} = R\trans Y R = \mtx{\widehat{Y}_{11} & \widehat{Y}_{12} \\ 
{\widehat{Y}_{12}}\trans & \widehat{Y}_{22}}$, 
where $\wh Y_{11}$ is $2m\x 2m$ and $\wh Y_{22}$ is $2r\x 2r$.  Evidently, 
\begin{eqnarray}
\label{eq3}
 \widehat{\Omega} \widehat{N}\widehat{Y}=\widehat{Y}\widehat{N}\widehat{\Omega}& \Leftrightarrow& \Omega NY=YN\Omega, \mbox{ and }
\\
\label{eq4} 
\widehat{N} \circ \widehat{Y} =O&\Leftrightarrow& N\circ Y = O.
\end{eqnarray}

Note that  $\widehat{\Omega} \widehat{N}\widehat{Y}=\widehat{Y}\widehat{N}\widehat{\Omega}$ is equivalent to
\begin{eqnarray}
\label{eq5}
\Omega_m P\widehat{Y}_{11} &=& 
\widehat{Y}_{11}P\Omega_m, \\
\label{eq6}
\Omega_r Q\widehat{Y}_{22}  &= &  
\widehat{Y}_{22}Q\Omega_r, \mbox{ and }\\
\label{eq7} 
 \Omega_m P \widehat{Y}_{12} & = &  \widehat{Y}_{12} Q \Omega_r.  
\end{eqnarray}

Equation (\ref{eq7}) 
is equivalent to $(\Omega_m P) \widehat{Y}_{12}  -  \widehat{Y}_{12} (Q \Omega_r)=O$.
  Also $\spspec(P) \cap \spspec(Q)= \emptyset\Leftrightarrow\spec(\Omega_m P)\cap \spec(\Omega_r Q)=\emptyset\Leftrightarrow\spec(\Omega_m P)\cap \spec(Q\Omega_r)=\emptyset$ since $\spec(Q \Omega_r)=\spec( \Omega_r Q)$.
 Sylvester's Theorem (see, e.g.,  \cite[Theorem 2.4.4.1]{HJ}) implies $\widehat{Y}_{12}=O$ is the only solution to \eqref{eq7} if and only if $\spec(\Omega_m P)\cap \spec(Q\Omega_r)=\emptyset$, i.e., if and only if $\spspec(P) \cap \spspec(Q)= \emptyset$.

First,  we  assume  $P$ and $Q$ have the SSSP and  $\spspec(P) \cap \spspec(Q)= \emptyset$   and show that $N$ has the SSSP. 
Since $\wh N\circ \wh Y=O$,  $P\circ \wh Y_{11}=O$. Together with \eqref{eq5} and the fact that $P$ has the SSSP, we conclude that $\widehat{Y}_{11}=O$. Similarly, 
$\widehat{Y}_{22}=O$.  Since $\spspec(P) \cap \spspec(Q)\ne \emptyset$, then as noted above
$\widehat{Y}_{12}=O$ is the only solution to \eqref{eq7}. Hence $\widehat{Y}=O$, and $Y=O$. So $N$ has the SSSP.

Now we assume  that $N$ has the SSSP and show that $P$ and $Q$ have the SSSP and  $\spspec(P) \cap \spspec(Q)= \emptyset$. 
 Suppose for some symmetric $m\x m$ matrix $\wh Y_{11}$ we have $P\circ \wh Y_{11}=O$ and $\Omega_m P\widehat{Y}_{11} = 
\widehat{Y}_{11}P\Omega_m$. Let $\wh X=\mtx{\widehat{Y}_{11} & O\\O & O}$.  Then $\wh N\circ \wh X=O$ and  $\widehat{\Omega} \widehat{N}\widehat{X}=\widehat{X}\widehat{N}\widehat{\Omega}$.  Since $N$ has the SSSP, this implies $\wh X=O$ by \eqref{eq3} and 
\eqref{eq4}.  Thus $\wh Y_{11}=O$ and $P$ has the SSSP.  Similarly, $Q$ has the SSSP.   If $\spspec(P) \cap \spspec(Q)\ne \emptyset$, then as noted above there exists a nonzero solution $\wh Y_{12}$ to \eqref{eq7}.  Then $\wh Y=\mtx{O&\wh Y_{12}\\\wh Y_{12}\trans & O}$ is a solution to $\widehat{\Omega} \widehat{N}\widehat{Y}=\widehat{Y}\widehat{N}\widehat{\Omega}$ and $\wh N\circ \wh Y$, contradicting that $N$ has the SSSP. 
\end{proof}

\begin{cor}\label{symplectic_disj_union}  Let $\GC G C$ be a coupled graph, let  $G_1$ be a subgraph respected by $\PP$, and let $G_2$ be the subgraph of $G$ induced by the complement  of $V(G_1)$.  Denote the order of $G$ by $2p$ and the order of $G_1$ by $2m$, and let $r=p-m$.
\ben[$(a)$]
\item\label{Ssubgarph1}   If $G=G_1\du G_2$ (i.e., $G_1$ is also an induced subgraph of $G$ and there are no edges between $G_1$ and $G_2$) and $G_i^{\PP|_{G_i}}$ allows symplectic spectrum $\LS_i$ for $i=1,2$, then $\GC G C$ allows the symplectic spectrum $\LS_1\cup\LS_2$. 

If $m=r=1$ and $G=G_1\du G_2$, then $G$ allows all symplectic eigenvalues equal and  is symplectic spectrally arbitrary.

\item \label{Ssubgarph2}  If  $G_1^{\PP|_{G_1}}$ allows symplectic spectrum $\LS_1$ with the SSSP, then $\GC G C$ allows allows the symplectic spectrum $\LS_1\cup \{d_1,\dots,d_r\}$ with SSSP for any distinct positive real numbers $d_i, i=1,\dots,r$ such that $d_i\not\in\LS_1$ for $i=1,\dots,r$.
\een
  \end{cor}
\bpf Choose   a representative labeling $L$  of $\PP$ such that  $G_1$ is labeled by  $\alpha_1=\{1,\dots,m\}\cup\{p+1,\dots,p+m\}$.  Observe that $\PP$ respects $G_2$ and  that  $G_2$ is labeled by $\alpha_2=\{m+1,\dots,m+r\}\cup\{p+m+1,\dots,2p\}$.
Let $L_i$ denote the labeling $L$ restricted to $G_i$. 

\eqref{Ssubgarph1}: Assume the hypotheses.   
Choose $N_i\in\symp(G_i^{L_i})$ with $\spspec(N_i)=\LS_i$ and define $N\in\symp(G^L)$ such that $N[\alpha_i]=N_i$ and all other entries are zero.  Then \[\spspec(N)=\spspec(N_1)\cup\spspec(N_2)=\LS_1\cup\LS_2\] by \Cref{directsum}.  The second statement is immediate from the first statement and \Cref{o4sympPD3SSA}.

\eqref{Ssubgarph2}: 
Let $N_1\in\symp(G_1^{L_1})$ with $\spspec(N_1)=\LS_1$ and the SSSP, and let $D$ be an $r\times r$ diagonal matrix with distinct positive diagonal entries $d_1,\dots,d_r$ such that $d_i\not\in\LS_1$ for $i=1,\dots,r$.
Define $\hat G^L$ to be the labeled graph obtained from $G^L$ by deleting all edges except those  in $E(G_1)$. Define   $\hat N\in\symp(\hat G^L)$ such that $\hat N[\alpha_1]=N_1$, $\hat N[\alpha_2]=D\oplus D$, and all other entries are zero. Then $\spspec(\hat N)=\LS_1\cup \{d_1,\dots,d_r\}$ and $\hat N$ has the SSSP by \Cref{directsum}. 
Since $\hat N$ has the SSSP and $G^L$ is a supergraph of $\hat G^L$, there is some $N\in\symp(G^L)$ with $\spspec(N)=\LS_1\cup \{d_1,\dots,d_r\}$ that has the SSSP.  
\epf

\Cref{symplectic_disj_union}
\eqref{Ssubgarph1} 
can be used to study certain labeled graphs 
that are  disjoint unions of labeled graphs.

\begin{ex}\label{ex:2K1K2}
    Consider the graph $2K_1\du K_2$ where the isolated vertices are named $v_1$ and $v_3$ and the vertices of $K_2$ are named $v_2$ and $v_4$.  By graph symmetry, we need to consider only two couplings, $\PP_2=\{(v_1,v_3), (v_2,v_4)\}$, which respects $2K_1$ and $K_2$, and a coupling such as $\PP_1=\{(v_1,v_2), (v_3,v_4)\}$ where an isolated vertex is coupled with a vertex of $K_2$.
     By \Cref{symplectic_disj_union}\eqref{Ssubgarph1}, $G^{\PP_2}$ allows all symplectic spectra. A labeling associated with $\PP_1$ requires simple symplectic eigenvalues by \Cref{iso-isol}.
\end{ex}

\begin{ex}\label{ex:2K2}
    Consider the graph $2K_2$ where the the vertices of one $K_2$ are named $v_1$ and $v_3$ and of the other $K_2$ are named $v_2$ and $v_4$.  By graph symmetry, we need to consider only two couplings, $\PP_2=\{(v_1,v_3), (v_2,v_4)\}$, which respects $K_2$ and $K_2$, and a coupling such as $\PP_1=\{(v_1,v_2), (v_3,v_4)\}$ where a vertex of one $K_2$ is coupled with a vertex of the other $K_2$.
For every labeling $L$ associated with $\PP_2$, $(K_2\du K_2)^L$  is symplectic spectrally arbitrary by \Cref{symplectic_disj_union}\eqref{Ssubgarph1}.

 Now consider the coupling $\PP_1=\{(v_1,v_2), (v_3,v_4)\}$ with labeling $L$ defined by {$L(v_1)= 1,  L(v_2)= 3, L(v_3)= 2,L(v_4)=4$}.  
 This is $(K_2\du K_2)^I$, which is symplectic spectrally arbitrary by \Cref{random2}.
 So every labeling of $K_2\du K_2$ is symplectic spectrally arbitrary.
\end{ex}

 The two previous examples complete the solution to the ISEP-$G$ for graphs of order four.
\Cref{symplectic_disj_union}
\eqref{Ssubgarph2} can be used  to provide forbidden families of labeled subgraphs for certain symplectic spectral properties. 

\begin{ex}
\label{ex:P6}
Consider the  matrix 
\[ 
N=
\renewcommand{\arraystretch}{1.25}
\left[\begin{array}{rrrrrr}
1 & \frac{1}{2} & 0 & 0 & 0 & 0 \\
\frac{1}{2} & \frac{5}{4} & \frac{1}{2} & 0 & 0 & 0 \\
0 & \frac{1}{2} & \frac{24}{25} & \frac{1}{10} & 0 & 0 \\
0 & 0 & \frac{1}{10} & 1 & -\frac{1}{2} & 0 \\
0 & 0 & 0 & -\frac{1}{2} & \frac{5}{4} & -\frac{1}{2} \\
0 & 0 & 0 & 0 & -\frac{1}{2} & 1
\end{array}\right]. \]
It can be verified that $N$ is positive definite 
and has  the SSSP.  {Let $\II$ denote a coupling of a path graph such that labeling the vertices  in path order is a representative labeling, so the coupled graph  of $N$ is $P_6^\II$.} 
The matrix $\Omega N$ has characteristic polynomial $(x^2 + 209/400)(x^2 + 1)^2$.
Hence $N$ has symplectic eigenvalues $1$, $1$, and
 $\sqrt{209/400}$. In particular, $N$ has a multiple symplectic eigenvalue. 
\end{ex}  

The following is a consequence of \Cref{symplectic_disj_union}, 
\Cref{ex:paw}, \Cref{ex:P6}, \Cref{K22cup2},  {\Cref{K4-ecup2},}  and \Cref{random}.

\begin{thm}\label{forbid4simple}
  Let $\GC G C$ be a coupled graph on $2p$ vertices such that 
 $G_1$ is a subgraph respected by $\PP$ on $2k$ vertices. 
 If $G_1$ is any of the following, then there is a positive definite matrix
  with labeled graph  representing $\GC GC$ with a non-simple symplectic eigenvalue:

\ben[\rm $(a)$]
\item A paw whose pendent edge does not join a vertex $v$ with $\mathfrak{c}({v})$.
\item $P_6^{\II}$.
\item A $4$-cycle {(any coupling)}.
\item $K_4-e$ (any coupling).
\item A complete graph of order at least four {(any coupling)}.
\een 
\end{thm}


\section{Coupled Zero Forcing}
\label{s:zeroforcing}

 Zero forcing and its variants such as positive definite and skew zero forcing  provide well-known  combinatorial upper bounds for the 
maximum nullities of the relevant types of  matrices described by  {a} graph, beginning with standard zero forcing  and {maximum nullity, or equivalently, maximum eigenvalue multiplicity for the standard} IEP-$G$ \cite{AIM08} 
(see  \cite[Chapter 9]{HLSbook} for discussion of variants).   In this section, we define a new form of zero forcing for coupled graphs as an upper bound on the maximum multiplicity of a symplectic eigenvalue among matrices associated with the coupled graph.  We use coupled zero forcing and prior results to solve the ISEP-$G$ for several families of coupled graphs. We also characterize coupled graphs that have coupled zero forcing number equal to one.

We begin by recalling the definition of the standard maximum nullity  and defining the maximum symplectic eigenvalue multiplicity   for a coupled graph.
Let $G$ be an unlabeled graph. The \emph{standard maximum nullity} of $G$ is $\M(G)=\max\{\nul A \colon A\in\sym(G)\}$.
Since the nullity of $A-\lam I$ is the multiplicity of $\lam$ as an eigenvalue of $A$ and $\G(A-\lam I)=\G(A)$, $\M(G)$ is the maximum multiplicity of an eigenvalue of $A\in\sym(G)$.

\begin{defn}[Maximum symplectic multiplicity]
    Let $\GC G C$ be a coupled graph.  
    The \emph{maximum symplectic multiplicity of $\GC G C$}, denoted by $\MC(\GC G C )$, is the maximum multiplicity among the symplectic eigenvalues of matrices $N\in\symp(G^\PP)$. 
\end{defn}

 Let $G$ be an unlabeled graph. The \emph{standard color change   rule} allows a blue vertex $v$ to force a white vertex $w$ to blue if $w$ is the only white vertex in $N_{G}(v)$. 
A set $B\subseteq V(G)$ of initially blue vertices is called a \emph{standard zero forcing set} of $G$ if repeated application of the standard zero forcing color change rule eventually colors every vertex of $G$ blue. The \emph{standard zero forcing number} of $G$, denoted by $\Z(G)$, is the minimum cardinality of a standard zero forcing set of $G$.

\begin{defn}[Coupled zero forcing] 
Let $\GC G C$ be a coupled graph, and suppose each vertex is colored either blue or white. The \emph{coupled zero forcing color change rule} consists of the following two operations:
\begin{enumerate}
    \item A blue vertex $v$ forces a white vertex $w$ to blue, provided that $w$ is the only white vertex in $N_G(v)\cup \{\pp(v)\}$.
    \item A white vertex $v$ forces itself to blue, provided that
    every vertex in $N_G(v)\cup \{\pp(v)\}$ is blue. 
\end{enumerate}
 A subset of  $V(G)$ of initially blue vertices is called a \emph{coupled zero forcing set} of $\GC G C$ if repeated application of the coupled zero forcing color change rule eventually colors every vertex of $\GC G C$ blue.  The \emph{coupled zero forcing number} of $\GC G C$, denoted by $\ZC(\GC G C)$, is the minimum cardinality of a coupled zero forcing set of $\GC G C$. 
\end{defn}

 To prepare for the proof that   $\MC(\GC G C)\le \ZC(\GC G C)$, we make a connection with the existing parameter loop zero forcing number, $\Z_\ell(G)$.  Once that is done (see \Cref{ZCZell}), it follows from known results about  loop zero forcing in \cite{param} that the final coloring of a coupled zero forcing set is independent of the forces chosen. 
 The {\em  loop zero forcing number} of a (simple unlabeled) graph $G$,  denoted by $\Zell(G)$, is  the zero forcing number defined by the $\Zell$-color change rule \cite{param}:
\ben
\item {(standard force)}  If $u$ is blue and exactly one  neighbor $w$ of $u$  is white, then change the color of $w$ to blue.  
\item {(white-vertex force)} If $w$ is white,  $w$ is not an isolated vertex, and every neighbor of $w$ is blue, then change the color of $w$ to blue.
\een

\begin{defn}
Given a coupled graph $\GC G C$, define the simple unlabeled graph $G(\PP)$ by starting with $G$  and  adding an edge (if not already present) between $v$ and $\pp(v)$ for each pair of coupled vertices. 
\end{defn}

   Observe that for any coupled graph $\GC G C$, the graph $G(\PP)$ always contains a perfect matching defined by the coupling, i.e., the edges  with endpoints $v$ and $\pp(v)$.   If $G$ has a perfect matching $M$, then $M$ naturally defines a coupling, called the \emph{coupling defined by $M$} and denoted by $\MM$, by choosing the pairs of endpoints of the edges in the matching  as the coupling.

\begin{rem}\label{ZCZell} {Since  $G(\PP)$  contains a perfect matching, it}  has no isolated vertices. Thus applying the coupled zero forcing color change rule to $\GC G C$ is  the same as applying the $\Zell$-color change rule to $G(\PP).$
    In particular, if $v$ and $\pp(v)$ are already adjacent in $G$ for every vertex $v$, then $G(\PP)=G$ { and $\Z_{C}(G^\PP)=\Zell(G)$}. 
\end{rem}

The next result, which plays an important role in the proof that $\MC(\GC G C) \leq \ZC(\GC G C)$, is also of independent interest, so we state it separately.
The geometric multiplicity of $\lam$ as a standard eigenvalue  $A\in\Rnn$ is denoted by $\mult_A(\lam) $.
In this section,  for $1\le i\le p$  
 we use $\pp(i)$ to denote $i+p$, and  for $p+1\le i\le 2p$ we use $\pp(i)$ to denote $i-p$.  
 
\begin{lem}\label{lem:coupledgraphnullity}
    Let $\GC G C$ be a coupled graph, let
    $N=[n_{uv}] \in\symp(G^{\PP})$,  and let $\mu$ be a standard  eigenvalue of $\Omega N$. Then 
    \ben[$(1)$]
    \item $\ker(\mu I-\Omega N)=\ker(\mu \Omega + N)$ and $\mult_{\Omega  N}(\mu)=\nul(\mu I- \Omega  N)=\nul(\mu \Omega + N)$.
  \item   $\G(\mu \Omega + N)=G(\PP)$ and every diagonal entry of $\mu \Omega + N$ is positive. 
    \een
    \end{lem}
\begin{proof} 
         The geometric multiplicity of $\mu$ as an eigenvalue of $\Omega N$ equals $\dim \ker(\mu I-\Omega N)$, i.e.,          $\mult_\mu( \Omega  N)=\nul(\mu I- \Omega  N)$.  
Multiplying $\mu I - \Omega N$ on the left by $\Omega$ and using $\Omega^2 = -I$ yields $\mu \Omega + N$.  Since $\Omega$ is invertible, $\ker(\mu I- \Omega  N)=\ker(\mu \Omega + N)$ and $\nul(\mu I- \Omega  N)=\nul(\mu \Omega + N)$. 

  Set $N' = \mu \Omega + N$ and $N'=[n'_{ij}]$.    For every $1\le i,j\le 2p$ such that $j\ne\pp(i)$, $n'_{ij}=n_{ij}$.  Furthermore, $n'_{i,\pp(i)} \neq 0$ for every $i$  since $N$ is a real positive definite matrix and  $\mu$ is a  nonzero  purely imaginary number by  \Cref{sympev}.  Thus $\G(\mu \Omega + N)=G(\PP)$.  Since $n'_{ii} =n_{ii}$, every diagonal element of $N'$ is positive.
\end{proof}

\begin{thm}\label{prop_zer_forc_vs_mult}
    Let $\GC G C$ be a coupled graph. Then $\MC(\GC G C) \leq \ZC(\GC G C)$.
\end{thm}
\begin{proof} 
    Let $N\in\symp(G^\PP)$ 
    and let $B$ be any coupled zero forcing set of $\GC G C$. 
    Suppose $\mu$ is an eigenvalue of $\Omega N$ with eigenvector $\bx=[x_i]$ such that $x_k = 0$ for all $k\in B$. 
     Set $N' = \mu \Omega + N$ and $N'=[n'_{ij}]$, so $\bx \in \ker(N')$, $\G(N')=G(\PP)$,  and every diagonal entry of $ N'$ is positive by \Cref{lem:coupledgraphnullity}. 
    Because an application of the coupled zero forcing color change rule corresponds to  showing another entry in  $\bx\in\ker(N')$ is zero, repeated application shows that  $\bx$ is the zero vector.  Thus,  $\ker(\mu I- \Omega N)$ does not contain a nonzero vector with $x_k=0$ for all $k\in B$. Hence, the  multiplicity of $\mu$ as an eigenvalue of $\Omega N$ is at most $|B|$
    (see Lemma 2.13 of \cite{HLSbook}).   By Proposition \ref{sympev}, the multiplicity of $\mu$ as an eigenvalue of $\Omega N$  equals the multiplicity of the symplectic eigenvalue $|\mu|$ of $N$. Since $\mu$ and $N$ are arbitrary, we have 
     $\MC(\GC G C) \leq \ZC(\GC G C)$.
\end{proof}

\begin{rem}
    Note that  the proof given in \Cref{lem:coupledgraphnullity}, which is necessary for the proof presented for \Cref{prop_zer_forc_vs_mult}, is specific to the real matrices $N$, because it relies on the fact that the sum of a real number and a nonzero purely imaginary number is necessarily nonzero. This is different from   the situation with standard maximum nullity and zero forcing, where $\M^F(G)\le\Z(G)$ holds regardless of the field $F$ in which the matrices reside
    (whether $\M^F(G)$ can be interpreted as the maximum multiplicity of an eigenvalue rather than as the maximum nullity {may}  vary with the field, as {may}   the value of $\M^F(G)$ for a fixed graph $G$.)
\end{rem}

\begin{ex}\label{example_path_simple}
    Let us consider the path graph $P_{2p}$ with vertices  named  in path order, i.e., $v_j$ is adjacent to  $v_{j+1}$ for all $j \in \{1,\ldots,2p-1\}$. Consider the coupling ${\MM}= \{(v_1,v_2),(v_3,v_4),\ldots$, $ (v_{2p-1},v_{2p})\}$. Then it is easy to see that $\{v_1\}$ is a coupled zero forcing set, and hence $\ZC(P_{2p}^{\MM}) = 1$. \Cref{prop_zer_forc_vs_mult} implies that $\MC(P_{2p}^{\MM}) = 1$. Therefore, the path graph $P_{2p}$ with coupling $\MM$ requires all  symplectic eigenvalues to be simple. 
\end{ex}

\begin{obs}\label{obs-delta-Zell}
    It is immediate from the definition that $\delta(G)\le \Zell(G) {\le \Z(G)}$.  { If $v$ and $\pp(v)$ are already adjacent in $G$ for every vertex $v$, then  $\delta(G) \le  \Z_{C}(G{^\PP})\le \Z(G)$.}
\end{obs}

\begin{ex} 
    Let us consider the  cycle graph $C_{2p}$  with vertex names  in cycle order, i.e., $v_j$ is adjacent to $v_{j+1}$ for all $j \in \{1,\ldots,2p-1\}$ and $v_{2p}$ is adjacent to $v_1$. 
    Consider the coupling ${ \MM}=\{(v_1,v_2),(v_3,v_4),\ldots,(v_{2p-1},v_{2p})\}$.  Then  $C_{2p}(\MM)= C_{2p}$ and $2=\delta(C_{2p})\le \ZC (C_{2p}^{\MM}) =\Zell(C_{2p})\le \Z(C_{2p})= 2$, 
    and  \Cref{prop_zer_forc_vs_mult} implies that $ \MC(C_{2p}^\MM) \leq 2$. 
\end{ex}

   Next we apply coupled zero forcing to solve the ISEP-$G$ for some additional families of graphs.  The \emph{corona} of $H$ with $K_1$, denoted by $H\circ K_1$, is formed from a graph $H$ by joining a new vertex $u_v$ to $v$ for each vertex $v$ of $H$.    We begin by developing some additional tools for graphs of the form $H\circ K_1$.
Let $H$ be a graph of order $p$.  Observe that $H\circ K_1$ has a unique perfect matching in which each leaf $u_v$ is matched with its neighbor $v\in V(H)$. Let $M$ be a representative labeling of the coupling $\MM$ defined by the perfect matching in which the leaves of $G$ are labeled $1,\dots,p$.

\begin{lem}\label{HcircK1}
     Let $H$ be a graph of order $p$ and let $\GC G M=(H\circ K_1)^{\MM}$. Then: 
     \ben[$(a)$]
     \item\label{HcircK1Z}   $\ZC(\GC G M)=\Zell(G)= \Z(H)$. 
     \item\label{HcircK1N}
    $N=\mtx{D & E\\E & A}\in \symp(G^{M})$ if and only if
    $
    A\in\sym(H^{M})$,  {where $H^M$ has the labeling induced by the labeling of $G^M$}, $D$ and $E$ are diagonal matrices, $D$ is positive definite, and each standard eigenvalue of $\sqrt{D}A\sqrt{D}-E^2$
     is a positive real. 
    In this case, $\spspec(N)=\{\sqrt{\lam_1},\dots,\sqrt{\lam_p}\}$, where $\spec (\sqrt{D}A\sqrt{D}-E^2)=\{\lam_1,\dots,\lam_p\}$; note $\sqrt{D}A\sqrt{D} -E^2\in   \sym(H^M)$. 
    
    \item \label{HcircK1oml} If  $H$ allows  every standard spectrum with ordered multiplicity list $\oml$, then  the coupled graph $\GC G M=(H\circ K_1)^{\MM}$ allows  every symplectic spectrum with ordered multiplicity list $\oml$.\een
\end{lem}
\bpf
\eqref{HcircK1Z}: From the definition of $\MM$, we see that $G(\MM)=G$, so $\ZC(\GC G M)=\Zell(G)$. Any standard zero forcing set for $H$ is   a $\Zell$-forcing set for $G$, so $\Zell(G)\le \Z(H)$. Suppose $B$ is a minimum loop forcing set for $G$.   If there exists $\ell\in B\setminus V(H)$, then $(B\setminus\{\ell\})\cup \{u_\ell\}$  is a loop forcing set for $G$ of the same size as $B$.  Repeating this process as needed gives    a minimum loop forcing set $\hat B$ for $G$ with $\hat B\subseteq V(H)$.  
Consider the forces that color  all vertices of $G$ blue: Since each vertex $u$ of $H$ has a white neighbor that cannot be forced until $u$ is blue,  all forces within $H$ are standard forces.  Thus $\hat B$ is a standard zero forcing set for $H$ and $\Z(H)\le\Zell(G)$.

\eqref{HcircK1N}: 
Let $N=\mtx{D & E\\E & A}$ be a  symmetric matrix  where $D$ and $E$ are  $p\times p$ diagonal matrices, and $A$ is a symmetric $p\x p$ matrix. It is immediate that 
$\G^{L}(N)=(H\circ K_1)^{M}$ if and only if $\G^L(A)=H^{M}$, and that if $N$ is positive definite 
then so is $D$. So we may assume that $D$ is positive definite and $\G^L(A)=H^{M}$.

Note that
$S=\begin{bmatrix}
D^{-1/2} &-D^{-1/2}E \\
O & D^{1/2} 
\end{bmatrix}
$
is  symplectic and 
\begin{equation}
\label{neweq} 
S\trans NS =\begin{bmatrix}
I  & O\\
O & \sqrt{D}A\sqrt{D}-E^2
\end{bmatrix}.
\end{equation} 
Thus $N$ is positive definite if and only if each (standard) eigenvalue of $\sqrt{D}A\sqrt{D}-E^2$ is a positive real. 
{By \Cref{ex:directsum}, when $N$ is positive definite, 
the symplectic eigenvalues of $S\trans NS$ (and hence of $N$) are $ \sqrt{\lambda_i}$ where 
$\lambda_1, \ldots, \lambda_p$ are the standard eigenvalues of $\sqrt{D}A\sqrt{D} -E^2$. 
}

\eqref{HcircK1oml}: Assume $H$ allows  every standard spectrum with ordered multiplicity list $\oml =(m_1,\dots,m_q)$. Given  positive real numbers $\nu_1<\dots<\nu_q$,  we can choose a matrix $A\in \sym(H)$ such that  the (standard) eigenvalues of $A$ are $\nu_i^2+1$ with multiplicity $m_i$ for $i=1,\dots,q$.  For $N=\mtx{I & I\\I & A}\in \symp(G^{M})$,  the symplectic eigenvalues of $N$ are  $\nu_i$ with multiplicity $m_i$ for $i=1,\dots,q$ by \eqref{HcircK1N}. 
\epf

\begin{prop}
\label{HcircK1MZ}
    Let $H$ be a graph  that  allows  every standard spectrum with multiplicity at most $\M(H)$. Then the coupled graph $\GC G M=(H\circ K_1)^{\MM}$ allows  every symplectic spectrum with multiplicity at most $\M(H)$.
    If in addition $\M(H)=\Z(H)$, then\[\M(H)=\MC(G^{\MM})=\ZC(G^{\MM})=\Zell(G)=\Z(H).\]
\end{prop}
\bpf 
Since $H$ allows every ordered multiplicity list such that the maximum multiplicity is at most $\M(H)$, the first statement follows from \Cref{HcircK1}\eqref{HcircK1oml}. 
This implies $\MC(G^{\MM})\ge \M(H)$. Since  $\MC(G^{\MM})\le \ZC(G^{\MM})$ by \Cref{prop_zer_forc_vs_mult} and $\ZC(G^{\MM})=\Zell(G)
 = \Z(H)$  by \Cref{HcircK1}\eqref{HcircK1Z},  the hypothesis $\M(H)=\Z(H)$ completes the proof.
\epf

The solutions to  the ISEP-$G$ for graphs of the form  $P_p\circ K_1$ (called  \emph{combs}) and for graphs of the form  $K_p\circ K_1$  are now immediate from \Cref{HcircK1MZ} and the following well-known facts: A path of order $p$ allows as a spectrum every set of $p$ distinct real numbers (see \cite{Hochstadt67a} or \cite[Theorem 1.15]{HLSbook}).
The complete graph $K_p$ allows as a spectrum every multi-set of $p$ real numbers that contains at least two distinct numbers  (see \cite{BLMNSSSY13} or \cite[Proposition 1.11]{HLSbook}).

\begin{cor}\label{PpcircK1}
     The coupled graph $(P_p\circ K_1)^{\MM}$ allows  every distinct symplectic spectrum and does not allow a multiple symplectic eigenvalue.  Furthermore, $\M(P_p)=\MC((P_p\circ K_1)^{\MM})=\ZC((P_p\circ K_1)^{\MM})=\Zell(P_p\circ K_1)=\Z(P_p)=1$.
\end{cor}

\begin{cor}\label{KpcircK1}
     The coupled graph $(K_p\circ K_1)^{\MM}$ allows  every symplectic spectrum with at least two distinct symplectic eigenvalues and does not allow a sympPD matrix.  Furthermore, $\M(K_p)=\MC((K_p\circ K_1)^{\MM})=\ZC((K_p\circ K_1)^{\MM})=\Zell(K_p\circ K_1)=\Z(K_p)=p-1$.
\end{cor}

\begin{thm}{\rm \cite{Ferg80, FF09}} A multi-set of real numbers $\lambda_1\le \lambda_2\le \dots\le \lambda_n$ is the spectrum of an $n\x n$  matrix  $A\in\sym(C_n)$  
if and only if
one of the following two conditions holds:
\ben[$(a)$]
\item  $\lambda_{n}>\lambda_{n-1}\ge \lambda_{n-2} > \lambda_{n-3} \ge\cdots $. 
\item $\lambda_{n}\ge \lambda_{n-1}> \lambda_{n-2} \ge \lambda_{n-3} >\cdots $.
\een
\end{thm}

To align with the results we wish to apply, we restate the previous theorem in terms of ordered multiplicity lists.

\begin{cor}  For a strictly increasing ordered set $\Lambda$ of $q$ real numbers  and 
 ordered multiplicity list $\oml=(m_1,\dots,m_q)$ where each $m_i$ equals one or two
 and  $n=\sum_{i=1}^q m_i$,  there exists $A\in\sym(C_n)$ with $\spec(A)=\Lambda$ if and only if $\sum_{i=k}^{k'}m_i$ is even
 for any $1\leq k< k'\leq q$ such that $m_k=m_{k'}=2$.
 
\end{cor}

\begin{cor}\label{CpcircK1}
The coupled graph $(C_p\circ K_1)^{\MM}$ allows  every symplectic spectrum that can be described by an 
 ordered multiplicity list $\oml=(m_1,\dots,m_q)$  where each $m_i$ equals one or two,
 $p=\sum_{i=1}^q m_i$, and
   $\sum_{i=k}^{k'}m_i$ is even for any $1\leq k< k'\leq q$ such that $m_k=m_{k'}=2$. No other symplectic spectra are allowed.
 Furthermore, $\M(C_p)=\MC((C_p\circ K_1)^{\MM})=\ZC((C_p\circ K_1)^{\MM})=\Zell(C_p\circ K_1)=\Z(C_p)=2$.
 \end{cor}

Corollaries \ref{PpcircK1}, \ref{KpcircK1} and \ref{CpcircK1} are only specific examples of a more general situation, described next. 

\begin{rem}\label{r:HcircK1}  Let $H$ is a graph for which the IEP-$G$ is solved
 by specifying which ordered multiplicity lists are permitted and all spectra are allowed for these lists.  Then  the ISEP-$G$ is solved by \Cref{HcircK1} for $G=(H\circ K_1)^\MM$  where $\MM$  couples each leaf with its unique neighbor:  The ordered symplectic multiplicity lists allowed for $(H\circ K_1)^\MM$ are exactly those allowed for $H$ and all possible ordered sets of positive real numbers are allowed with these multiplicities. \end{rem}

 A coupled graph $\GC G C$ is \emph{connected} if $G$ is connected.  A \emph{caterpillar} is a tree  $T$ that has a path {$P$} (called a \emph{central path})
 such that every other  vertex of $T$ is  adjacent to a vertex of $P$.  
A tree $T$ has at most one  perfect matching, and whether or not $T$ has a perfect matching can be determined by starting with a {leaf}, matching it {with its unique neighbor}, removing both matched vertices from $T$, and continuing in this manner.  A perfect matching is found if and only if the process ends with no vertices \cite{IMA10}.

\begin{thm}\label{t:ZC1} Let $G$ be  a  graph.  Then $\Zell(G)=1$ if and only if $G$ is a caterpillar.
Furthermore, $\ZC(\GC G C)=1$ if and only if  $G(\PP)$ is a caterpillar that contains a perfect matching.
 If $G$ is connected, then  $\ZC(\GC G C)=1$ if and only if $G$ is a caterpillar that contains a perfect matching $M$ and $\PP$ is the coupling defined by $M$.
\end{thm}
\bpf Let $T$ be a caterpillar.  Then the set 
consisting of an endpoint of a central path is a $\Zell$-forcing set, so $\Zell(T)=1$.  If $T$ has a perfect matching and  $\MM$ is the coupling defined by the perfect matching, then $T(\MM)=T$ and $\ZC(\GC T M)=\Zell(T)=1$.

Suppose  that $\Zell(G)=1$.  Then $G$ is connected. 
Since  $\Z_\ell(G)\ge \Z_+(G)$ \cite{param},  $G$ is a tree \cite{Zplus}.  We show $G$ has a central path and every vertex is on the central path {or adjacent to a vertex on the central path}.  Choose a minimum  $\Zell$-forcing set $\{u_1\}$ for $G$. Assume 
the forces so far {that are not white-vertex forces} are $u_1\to\dots\to u_k$. 
If every neighbor of $u_k$  has degree one, then each can force itself and $G$ is a caterpillar with central path $G[\{u_1,\dots,u_k\}]$.  If not, then 
every white neighbor of $u_k$ except  one has degree one  and so can force itself; we  denote by $u_{k+1}$ the one neighbor that does not have degree one 
and note that $u_k$ can now force $u_{k+1}$.
This process shows that $G$ is a caterpillar.

 If   $G(\PP)$  is not a caterpillar with a perfect matching, then $G(\PP)$ is not a caterpillar, so $\ZC(G^{\PP})=\Zell(G(\PP))>1$.

Finally, assume $G$ is connected and $\ZC(G^{\PP})=1$.  Then $G$ has at least $n-1$ edges and $G(\PP)$  is  a caterpillar with a perfect matching, implying $G$ is  a caterpillar with a perfect matching $M$ and  $\PP$ is the coupling defined by $M$.
\epf

\begin{rem}
    A caterpillar $T$  with  maximal central path $P$ has a perfect matching if and only if
    \ben
    \item Every vertex has degree at most three.
    \item If the degree three vertices of $T$ are deleted from $P$, the subpaths of $P$ that remain (if any) all have even order.
    \een
     A comb $P_p\circ K_1$ is an example of a caterpillar with a perfect matching. See \Cref{f:catmatch} for another example.
\end{rem}

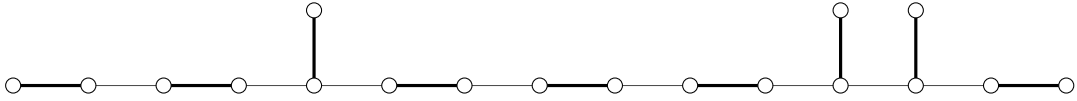
\begin{figure}[!ht]
    \centering
\begin{tikzpicture}
\draw[opacity = 0.7](0,0)--(14,0); \foreach\i in{4,11,12}{\draw(\i,0)--(\i,1);}
\foreach\i in{4,11,12}{\draw[very thick](\i,0)--(\i,1);} \foreach\i in{0,2,5,7,9,13}{\draw[very thick](\i,0)--(\i+1,0);}
\foreach\i in{0,1,...,14}{\draw[fill=white](\i,0)circle(0.1);} \foreach\i in{4,11,12}{\draw[fill=white](\i,1)circle(0.1);}

\end{tikzpicture}
\caption{A caterpillar with a perfect matching (black edges form a  perfect matching).}
\label{f:catmatch}
\end{figure}

 We can now solve the ISEP-$G$ for caterpillars with perfect matchings. The next result is immediate from \Cref{prop_zer_forc_vs_mult},  \Cref{AllSimple}, and \Cref{t:ZC1}.

\begin{cor}\label{c:ZCsimple} 
 Let $\GC G C$ be a coupled graph of order $2p$ such that $\MC(\GC G C)=1$.  Then  $\GC G C$
 allows  a multi-set of $p$ positive  real numbers as a symplectic spectrum if and only if the elements are distinct. In particular, this applies to every 
 caterpillar  that contains a perfect matching (with  the coupling  defined by  the perfect matching).  
\end{cor}

As is the case for (standard)  zero forcing and  maximum nullity, it is possible to have the maximum symplectic multiplicity less than the coupled zero forcing number of a coupled graph.  But the situation is actually somewhat different, because for a coupled graph of order $2p$,  $\MC(\GC G C)\le p$ and $\ZC(\GC G C)\le 2p-1$ {and both bounds are sharp as illustrated by $K_{2p}$: We saw in \Cref{random} that $\MC(K_{2p}^\PP)=p$ for any coupling, and} $\ZC(\GC {K_{2p}} C)= 2p-1$ (since $ \ZC(\GC {K_{2p}} C)=\Zell(K_{2p})= 2p-1$).  


\section{Concluding remarks}\label{s:conclude}

This work  establishes a new line of investigation --- the Inverse Symplectic Eigenvalue Problem.
We have completely resolved  the  ISEP-$G$  for at least one coupling of more families of graphs than the IEP-$G$ has been solved for, with the latter involving more than a dozen authors over more than twenty years.
 In particular, the following coupled graphs are symplectic spectrally arbitrary. 
\bit
\item $K_{2p}^\PP$ where $\PP$ is any coupling.
\item $(K_p\du K_p)^\II$  where $\II$ pairs the vertices of one $K_p$ with the vertices of the other $K_p$.
\item $(\OL{K_p}\vee K_p)^\MM$ where $\MM$ couples each vertex of $\OL{K_p}$ with a vertex in $K_p$.
\item $K_{p,p}^\MM$  where $\MM$ couples each vertex in one partite set with a vertex in the other partite set.
\item $\tp_{2p}^\II$. 
\eit 
Other graphs for which the ISEP-$G$ is solved: 
\bit
\item All graphs of order 4  with all couplings (see \Cref{s:order4} for details of what spectra are allowed).
\item Path $P_{2p}^\MM$ where $\MM$ is defined by the perfect matching (all eigenvalues distinct).
\item $T^\MM$ where  $T$ is a caterpillar with a perfect matching  and $\MM$ is defined by the perfect matching  (all eigenvalues distinct).
\item  $(K_p\circ K_1)^\MM$ where $\MM$  couples each leaf with its unique neighbor (any symplectic spectrum with at least two distinct symplectic eigenvalues).
\item Comb $(P_p\circ K_1)^\MM$ where $\MM$  couples each leaf with its unique neighbor (all eigenvalues distinct).
\item Sun $(C_p\circ K_1)^\MM$ where $\MM$  couples each leaf with its unique neighbor (see \Cref{CpcircK1} for details of what spectra are allowed).
\item $(H\circ K_1)^\MM$ where $\MM$  couples each leaf with its unique neighbor and the IEP-$G$ is solved by specifying which ordered multiplicity lists are permitted and all spectra are allowed for these lists (depends on the spectra allowed by $H$).
\eit

These solutions have been possible through an assortment of tools we have developed for the ISEP-$G$, including the Strong Symplectic 
Spectral Property, results on symplectic positive definite matrices, and coupled zero forcing. Since labeled graphs must be used for symplectic spectra, we have introduced couplings to manage symplectically equivalent labeling, thereby substantially reducing the number of labelings that need to be considered.

\appendix
\section{Appendix}\label{appendix}
 This section contains the proofs of the theoretical 
properties associated with the SSSP.

We begin by  developing a geometric  condition that arises in our proofs, but is  equivalent to the SSSP.
 Throughout this section, we view the set of real $n\times n$ matrices as an inner product space, where the 
inner product between such matrices $A$ and $B$ is given by $\mbox{tr}(A^\top B)$. Note that the set,  
$\symm(n)$, of $n\times n$ symmetric matrices
is a subspace, and for symmetric $A$ and $B$ the inner product simplifies to $\mbox{tr}(AB)$. 
We will be working with subspaces $U$ of Symm($n$), with \emph{orthogonal complement}  \[ U^{\perp}=\{ A\in \symm(n)\colon\mbox{tr}(AB)= 0 
\mbox{ for all $B \in U$} \}; \]   
that is, the orthogonal complement with respect to $\symm(n)$.

\begin{rem}
\label{rem:Gcomp}
Let $G^L$ be a labeled graph on $2p$ vertices and $N$ be a positive definite matrix 
with labeled graph $G^L$. The span $S$ of  $\sym(G^L)$ is the subspace of 
$\symm(2p)$ consisting of all $2p\times 2p$ matrices whose labeled graph
is a subgraph of $G^L$.  Thus, $S$ is spanned by the matrices  $E_{ij} + E_{ji}$ where $ij$ is an edge of $G^L$
and the matrices $E_{ii}$ for $i=1, \ldots, 2p$. 
 It follows that the $2p\times 2p$ symmetric matrix $Y$ is in $S^{\perp}$ if and only if 
$Y_{ij}=0$ whenever $i=j$ or $ij$ is an edge of $G^L$. Since $ N\in\symp(G^L)$, it follows 
that $Y \in S^{\perp}$ if and only if $N\circ Y=O$. 
\end{rem}

\begin{lem}\label{lem:orthocomp}
Let $N$ be a $2p \times 2p$ positive definite matrix and $\mathcal{U}=  \{NM+ M\trans N\colon M \in \spLA(2p)\}$.
Then 
$ \mathcal{U}^\perp= \{Y\in\symm(2p)\colon \Omega NY=YN\Omega\}$. 
\end{lem}

\bpf 
By definition, 
\[ 
\mathcal{U}^\perp=
\{ Y \in \symm(2p)\colon
\tr(Y(NM+M\trans N))=0 
\mbox{ for all $M \in \mathfrak{sp}(2p)$} \}.\] 
Note that 
\begin{eqnarray*}
\tr(Y(NM + M\trans N)) &= & \mbox{tr}(YNM + YM\trans N) \\
& = & \tr(YNM + NMY)\\
& = & 2\tr(YNM). 
\end{eqnarray*}
The second equality follows from the facts that   the trace of a matrix equals the trace of its transpose and  both  $N$ and $Y$ are symmetric. 
The third equality follows from the facts that $\tr(A+B)=\tr(A)+\tr(B)$ and 
$\mbox{tr}(AB)= \mbox{tr}(BA)$ for any appropriately sized matrices. 
Let $\mathcal{W}$ denote the set of $2p\times 2p$
matrices $W$  such that $\tr(WM)=0$ for all $M \in \spLA (2p)$.  
Then, $\mathcal{U}^\perp$
is the set of symmetric matrices $Y$ for which
$YN \in \mathcal{W}$.

We  determine  the form of
$
W= \begin{bmatrix}
    W_{11} & W_{12} \\
    W_{21} & W_{22}
\end{bmatrix}\in \mathcal{W}$, where each $W_{ij}$ is $p\times p$.
Using the block form for a Hamiltonian matrix $M$ in \eqref{Eq:HamForm}, we see that $W\in \mathcal{W}$ if and only if 
$
\mbox{tr} (W_{11}R  + W_{12}F + 
W_{21}E-W_{22}R^\top)=0 $ for all symmetric  $E$ and $F$, and all $R\in\R^{p\x p}$. As the choices of $R$, $E$,  and $F$ are independent, 
 $W \in \mathcal{W}$ if and only 
if $\mbox{tr}(W_{12}F)=0$ for all symmetric $F$, 
$\mbox{tr}(W_{21}E)=0$ for all symmetric $E$,
$\mbox{tr}((W_{11} - W_{22}^\top)R)=0$ for all $R$.
Thus, $W \in \mathcal{W}$ if and only if $W_{12}$ 
and $W_{21}$ are skew-symmetric, and  $W_{11}=W_{22}^\top$. 
One can verify that these  last conditions are equivalent to 
$ \Omega W^\top= W\Omega$. Therefore, 
$\mathcal{U}^\perp=\{ NM+ M\trans N\colon M \in \spLA(2p) \}$  is the set of symmetric matrices $Y$ such that $\Omega N Y= YN \Omega$.
\epf

\begin{thm}
\label{SSSPequiv}
Let $N$ be a $2p\times 2p$ real positive definite matrix.
Then  $N$ has the SSSP if and only if  equation (\rm \ref{SSSPsum}) holds. 
\end{thm}

\begin{proof}
Equation \eqref{SSSPsum} holds if and only if 
the orthogonal complements of the  two spaces in equation (\ref{SSSPsum}) intersect only in the zero matrix.  The theorem now follows from \Cref{rem:Gcomp} and \Cref{lem:orthocomp}.
\end{proof}  

\bigskip\noindent
{\bf Proof of \Cref{t:verification} (SSSP Verification Matrix)}. \ms

\noi
 Recall that  $N$ is a $2p\times 2p$ positive definite matrix with labeled graph $G^L$.
Note that $\Xi(N)$ is a $(2p^2  -  
p -|E(G)|) \times (2p^2+p)$
matrix. Thus, the equivalence of \eqref{b:verification} and \eqref{c:verification} follows from elementary linear algebra. 

Note that  
\[ 
\uvec^{-1} (\mbox{CS}( \Phi(N)))=\{ S^\top N+NS: S \in \spLA(2p)\}
,\] 
where $\mbox{CS}(\Phi(N))$  denotes the column space of
$\Phi(N)$. Also, $\uvec(\Span \mathcal{S}(G^L))$ is
the subspace of $\mathbb{R}^{2p^2+p}$ consisting of those vectors 
having zero coordinate in each row of $\Phi(N)$ that 
is a row included in $\Xi(N)$.
The equivalence of \eqref{a:verification}  and \eqref{b:verification}  now follows. 
\qed\ms

The next two proofs are  essentially the same as the proofs of the Supergraph Theorem  and Bifurcation Theorem for the IEP-$G$ with the 
functions $f$  and $g$ given below  suitably modified.
 In these, we use the word ``neighborhood'' in 
the topological sense. That is, a neighborhood of a point $p$ is an open set in the pertinent topological space that contains $p$. \ms

\medskip
\noi {\bf Proof of \Cref{supergraph} (Supergraph Theorem).}

\medskip
\noi
Recall  $G^L$ is a labeled graph of order $2p$,  $N\in\symp(G^L)$ has the SSSP, and   
 $H^L$ is  obtained from $G^L$ by inserting some edges.
Define the function $f \colon \spLA (2p) \times  \Span (\mathcal{S}(G^L))\rightarrow \symm(2p)$
by 
\[ f(S, B)= e^{S\trans} N e^S -B.\]
Note that $f(O,O)=N$ and the linearization of $f$ at $(O,O)$ 
is the function $L \colon \spLA (2p) \times  \Span (\mathcal{S}(G^L)) \rightarrow \symm(2p)$ 
such that
\[ 
L(S,B)= N + NS + S\trans N -B.
\]
As $N$ has the SSSP, $L$ is a surjection.
As $f$ is continuous, there exists a neighborhood $U$ containing $(O,O)$ 
such that for all $(S,B) \in U$, the  $(i,j)$-entry of $f(S,B)$ is nonzero whenever
the $(i,j)$-entry of $N$ is nonzero.  By the Inverse Function Theorem, there exists 
a neighborhood $V$ of $(O,O)$ contained in $U$ and a neighborhood $W$ of $N$ 
contained in $\symm(2p)$ such that $f(V)=W$.  Let $C$ be the adjacency matrix of the 
graph on $2p$ vertices whose edges are those in $H$ but not in $G$.  For sufficiently small $\epsilon$, $N + \epsilon C \in W$. Hence, there exists $(S,B) \in V$
such that $e^{S\trans}Ne^S-B= N + \epsilon C$.  Equivalently, 
$e^{S \trans} Ne^{S} = N +B +\epsilon C$.  As $S \in \spLA (2p)$, $e^S$ is a symplectic matrix, 
and hence $e^{{S} \trans} N e^{S} $ and $N$ have the same symplectic eigenvalues.  The choices  of 
$\epsilon$, $V$, $W$, and $C$ guarantee that the labeled graph of $N+B+\epsilon C$ is $H^L$.
Therefore, $e^{S{\trans}}N e^S$ is a matrix with labeled graph $H^L$ and the same symplectic eigenvalues as 
$N$. 
\qed

\medskip
\noi {\bf Proof of \Cref{Thm:bifurcation} (Bifurcation Theorem).}

\medskip\noi
Define the function $g\colon \spLA (2p) \times  \Span (\mathcal{S}(G^L)) \rightarrow \symm(2p)$ 
by: 
\[ g(S, B)= e^{S\trans} (N-B) e^S.\]
Note that $g(O,O)=N$ and the linearization of $g$ at $ (O,O)$
is the function $L\colon \spLA (2p) \times \Span (\mathcal{S}(G^L)) \rightarrow \symm(2p)$ 
such that
\[ 
L(S,B)= N + NS + S\trans N -B.\]
As $N$ has the SSSP, $L$ is a surjection.
As $g$ is continuous, there exists a neighborhood $U$ of $(O,O)$ 
such that for all $(S,B) \in U$, the  $(i,j)$-entry of $g(S,B)$ is nonzero whenever
the $(i,j)$-entry of $N$ is nonzero.  By the Inverse Function Theorem, there exists 
a neighborhood $V$ of $(O,O)$ contained in $U$ and a neighborhood $W$ of $N$ 
contained in $\symm(2p)$ such that $f(V)=W$. Take $\epsilon $ small enough so that 
the $\epsilon$ ball around $N$  is contained in $W$. Suppose that $\| \widehat{N} -N \|<\epsilon$.
Then there exists $(S,B) \in V$ such that $\widehat{N}= g(S,B)= e^{S\trans} (N-B) e^{S}$.
Hence $\widehat{N}$ and $(N-B)$ have the same symplectic spectrum, and $N-B$ has labeled graph $G^L$. 
\qed \ms

\bigskip
\noindent
{\bf Acknowledgments:} 
The authors thank the American Institute of Mathematics and the US National Science Foundation for support of the Inverse Eigenvalue Problem of a Graph - Zero Forcing AIM Research Community, where this research originated, and   the Pacific Institute for the Mathematical Sciences for partial support of H.~Gupta as a post-doctoral fellow. The authors thank the referees for very helpful reports that improved the exposition of the paper.

\end{document}